\documentclass[11pt]{amsart}

\usepackage[T1]{fontenc}
\usepackage{lmodern}
\usepackage{microtype}
\usepackage[margin=1.02in]{geometry}
\usepackage{amsmath,amssymb,mathtools}
\usepackage{aliascnt}
\usepackage{bm}
\usepackage{xcolor}
\usepackage[colorlinks=true,linkcolor=blue!45!black,citecolor=blue!45!black,urlcolor=blue!45!black]{hyperref}
\usepackage[nameinlink,capitalise,noabbrev]{cleveref}

\allowdisplaybreaks
\numberwithin{equation}{section}

\newtheorem{theorem}{Theorem}[section]
\newaliascnt{proposition}{theorem}
\newtheorem{proposition}[proposition]{Proposition}
\aliascntresetthe{proposition}
\newaliascnt{lemma}{theorem}
\newtheorem{lemma}[lemma]{Lemma}
\aliascntresetthe{lemma}
\newaliascnt{corollary}{theorem}
\newtheorem{corollary}[corollary]{Corollary}
\aliascntresetthe{corollary}
\theoremstyle{remark}
\newaliascnt{remark}{theorem}
\newtheorem{remark}[remark]{Remark}
\aliascntresetthe{remark}

\newcommand{\R}{\mathbb R}

\newcommand{\Z}{\mathbb Z}
\newcommand{\E}{\mathbb E}
\newcommand{\ip}[2]{\left\langle #1,#2\right\rangle}
\newcommand{\norm}[1]{\left\lVert #1\right\rVert}
\newcommand{\abs}[1]{\left\lvert #1\right\rvert}
\newcommand{\sgn}{\operatorname{sgn}}
\newcommand{\arsinh}{\operatorname{arsinh}}

\newcommand{\Id}{\mathrm I}
\newcommand{\iu}{\mathrm i}
\newcommand{\eps}{\varepsilon}
\newcommand{\BK}{B_{\mathrm K}}
\newcommand{\KG}{K_{\mathrm G}}
\newcommand{\Ko}{\mathfrak K_{\mathrm K}}
\newcommand{\KoDim}[1]{\mathfrak K_{\mathrm K,#1}}
\newcommand{\cK}{c_{\mathrm{Kri}}}
\newcommand{\Podd}{P_{\mathrm{odd}}}
\newcommand{\xvec}{\bm x}
\newcommand{\yvec}{\bm y}
\newcommand{\uvec}{\bm u}
\newcommand{\vvec}{\bm v}
\newcommand{\zvec}{\bm z}
\newcommand{\Xvec}{\bm X}
\newcommand{\Yvec}{\bm Y}
\newcommand{\avec}{\bm a}
\newcommand{\bvec}{\bm b}
\newcommand{\muvec}{\bm\mu}
\newcommand{\tauvec}{\bm\tau}
\newcommand{\sigmavec}{\bm\sigma}
\newcommand{\Gvec}{\bm G}
\newcommand{\dd}{\mathop{}\!\mathrm d}

\title{The K\"onig constant is one}

\author[X. Xie]{Xinyuan Xie}
\address{(X.X.) Department of Mathematics, University of California, Irvine, CA 92697, USA}
\email{xinyuax7@uci.edu}

\author[H. Zhang]{Haonan Zhang}
\address{(H.Z.) Department of Mathematics, University of South Carolina, Columbia, SC 29208, USA}
\email{haonanzhangmath@gmail.com}

\date{}

\hypersetup{
  pdftitle={The K\"onig Constant},
  pdfauthor={Xinyuan Xie and Haonan Zhang},
  pdfsubject={The Grothendieck inequality and the K\"onig's bilinear form},
  pdfkeywords={Grothendieck inequality, K\"onig's bilinear form, Krivine rounding, Gaussian partitions, oscillatory Gaussian integrals, cubic parity}
}

\begin{document}

\begin{abstract}
For each $N\geq1$, consider the normalized K\"onig bilinear form
$B_{\mathrm K}:L_\infty(\mathbb R^N)\times L_\infty(\mathbb R^N)\to\mathbb R$
given by
\[
B_{\mathrm K}(f,g):=\frac{1}{(\sqrt{2}\pi)^N}
\iint_{\mathbb R^N\times\mathbb R^N}
f(x)g(y)e^{-(\lVert x\rVert^2+\lVert y\rVert^2)/2}
\sin\langle x,y\rangle\,\mathrm d x\,\mathrm d y,
\]
We define the K\"onig constant by
\[
\mathfrak K_{\mathrm K}:=\sup_{N\geq1}\sup_{\substack{f,g:\mathbb R^N\to\{\pm1\}\\
f,g\ \mathrm{measurable}}}B_{\mathrm K}(f,g).
\]

The study of this bilinear form arose from efforts to determine the exact value of the Grothendieck constant.
K\"onig~\cite{KONIG}
conjectured that the sharp value should instead be given by the
one-dimensional half-spaces $B_{\mathrm K}(\operatorname{sgn}(x_1),\operatorname{sgn}(x_1))=\frac{2}{\pi}\log(1+\sqrt{2})$.
A positive answer to this conjecture, together with a classical upper bound of Krivine \cite{KRIVINE}, would determine the exact value of the Grothendieck constant. 
In a breakthrough~\cite{BMMN}, Braverman, Makarychev,
Makarychev, and Naor disproved K\"onig's conjecture already in dimension two and used their counterexamples to obtain the first
strict improvement over Krivine's bound. 
One question in \cite{BMMN} attempts to determine the Grothendieck
constant through alternating Krivine rounding schemes arising from K\"onig's bilinear form in high dimension. More recently, Li
et al.~\cite{LISK} constructed high-dimensional examples showing that $\mathfrak K_{\mathrm K}\ge 0.59357$.

An elementary Fourier argument gives $\mathfrak K_{\mathrm K}\le 1$ and excludes
equality for every finite-dimension. In this paper, we prove that $\mathfrak K_{\mathrm K}=1$ by
constructing a family of Boolean pairs in high dimensions. In particular, this gives a negative answer
to the high-dimensional aspect of the question in \cite{BMMN}.

\end{abstract}

\maketitle

\section{Introduction}

The (real) Grothendieck constant, denoted by $\KG$, is the smallest
$K>0$ such that, for every pair of positive integers $r,s$, every real
$r\times s$ matrix $(a_{ij})$, every real Hilbert space $\mathcal H$,
and every choice of unit vectors
$\xvec_1,\ldots,\xvec_r,\yvec_1,\ldots,\yvec_s\in\mathcal H$, one has
\begin{equation}\label{eq:GI}
 \left|\sum_{i=1}^r\sum_{j=1}^s a_{ij}\ip{\xvec_i}{\yvec_j}\right|\leq K\max_{\varepsilon_i,\delta_j\in\{\pm1\}}\left|\sum_{i=1}^r\sum_{j=1}^s a_{ij}\varepsilon_i\delta_j\right|.
\end{equation}
Grothendieck's theorem asserts that such a finite constant exists~\cite{GROTH}.
Equivalently, the same optimal constant is the least constant in the
continuous kernel form
of \cref{eq:GI}: for $\sigma$-finite measure spaces $(\Omega,\mu)$ and
$(\Lambda,\nu)$, every real $a\in L_1(\Omega\times\Lambda)$, and
measurable maps $u:\Omega\to\mathcal H$ and $v:\Lambda\to\mathcal H$
of norm at most one,
\begin{equation}\label{eq:GI-continuous}
 \abs{\iint a(s,t)\ip{u(s)}{v(t)}\,\dd\mu(s)\,\dd\nu(t)}
 \leq\KG\sup_{\substack{\varepsilon:\Omega\to\{\pm1\}\\
                          \delta:\Lambda\to\{\pm1\}}}
 \abs{\iint a(s,t)\varepsilon(s)\delta(t)\,\dd\mu(s)\,\dd\nu(t)}.
\end{equation}
Here the supremum is over measurable sign functions.
Grothendieck's inequality plays an essential role in many areas, such as functional analysis, Banach space theory, operator algebras,
semidefinite programming and approximation algorithms, combinatorial
optimization and computational complexity, and quantum information; see
\cite{PISIER,VERSHYNIN,ALONNAOR,KHOTNAOR,NRV,BMMN} and the references
therein. Its exact value remains unknown. Grothendieck's original
estimates~\cite{GROTH} were
\[
 \frac\pi2\leq\KG\leq\sinh\!\left(\frac\pi2\right).
\]

The two classical historical estimates are
\[
 K_{\mathrm{DR}}:=1.676956674215576\ldots\leq\KG
 \leq\cK^{-1}
 =\frac{\pi}{2\log(1+\sqrt2)}
 =1.782213978\ldots,
\]
where $K_{\mathrm{DR}}$ is the lower bound obtained independently by
Davie~\cite{DAVIE} and Reeds~\cite{REEDS}, and the upper bound is due
to Krivine~\cite{KRIVINE,KRIVINE79}. The best currently announced
bounds, due to Saha et al.~\cite{SLXCKKM}, are
\[
 \frac{6\pi}{11}=1.713595992\ldots\leq\KG
 \leq1.7818666069360661
 <\cK^{-1}-3.4737\cdot10^{-4}.
\]

Krivine's elegant rounding argument
\cite{KRIVINE,KRIVINE79} establishes the upper historical estimate.
Given unit vectors $\xvec_i,\yvec_j$, it seeks random signs whose
pairwise correlations are a fixed multiple of
$\ip{\xvec_i}{\yvec_j}$, so that averaging transfers the vector
inequality to the Boolean one. The basic partition is a Gaussian
hyperplane: if $k\geq1$, $\Gvec\sim\mathcal N(0,\Id_k)$, and
$\xvec,\yvec\in\R^k$ are unit vectors, then
\begin{equation}\label{eq:hyperplane-correlation}
 \E\bigl[\sgn\ip{\Gvec}{\xvec}\,\sgn\ip{\Gvec}{\yvec}\bigr]=\frac2\pi\arcsin\ip{\xvec}{\yvec}.
\end{equation}
Here $\sgn(0):=1$. This is the Gaussian arcsine law; it follows from
rotational invariance since the two signs disagree with probability
$\arccos\ip{\xvec}{\yvec}/\pi$.

Since this correlation is nonlinear, Krivine first preprocesses the
inner products. His tensor-power lift gives, after harmless padding,
unit vectors $\Phi_\gamma(\xvec)$ and $\Psi_\gamma(\yvec)$ satisfying
\[
 \ip{\Phi_\gamma(\xvec)}{\Psi_\gamma(\yvec)}
 =\sin\bigl(\gamma\ip{\xvec}{\yvec}\bigr).
\]
The absolute coefficient sum of this sine series is
$\sinh\gamma$, so the lift is available when $\sinh\gamma\leq1$.
Applying the hyperplane rule after the lift then yields the linear
correlation $(2\gamma/\pi)\ip{\xvec}{\yvec}$. The largest permitted
value is $\gamma=\arsinh(1)$, and hence
\begin{equation}\label{eq:cK}
 \cK=\frac2\pi\arsinh(1)=\frac2\pi\log(1+\sqrt2),
 \qquad
 \frac1{\cK}=\frac{\pi}{2\log(1+\sqrt2)}.
\end{equation}
This recovers Krivine's classical bound
\begin{equation}\label{eq:Krivine-bound}
 \KG\leq\frac1{\cK}.
\end{equation}

To determine the exact value of $\KG$, K\"onig~\cite{KONIG} considered
the following analytic extremal problem. For measurable Boolean functions
$f,g:\R^N\to\{\pm1\}$, define the normalized oscillatory Gaussian
functional
\begin{equation}\label{eq:BK-def}
 \BK(f,g):=\frac1{(\sqrt{2}\pi)^N}\iint_{\R^N\times\R^N}f(\xvec)g(\yvec)e^{-(\norm{\xvec}^2+\norm{\yvec}^2)/2}\sin\ip{\xvec}{\yvec}\,\dd\xvec\,\dd\yvec,
\end{equation}
for each $N\geq1$, define its fixed-dimensional K\"onig constant by
\begin{equation}\label{eq:KoN-def}
 \KoDim{N}:=\sup_{\substack{f,g:\R^N\to\{\pm1\}\\
                    f,g\ \mathrm{measurable}}}\BK(f,g),
\end{equation}
and define the K\"onig constant by
\begin{equation}\label{eq:Ko-def}
 \Ko:=\sup_{N\geq1}\KoDim{N}.
\end{equation}
Replacing $g$ by $-g$ shows that the same constants are obtained if
$\abs{\BK(f,g)}$ is used in the suprema. If $h_N(\xvec):=\sgn(x_1)$,
then the Gaussian half-space
identity gives
\begin{equation}\label{eq:halfspace-value-intro}
 \BK(h_N,h_N)=\cK
\end{equation}
for every $N$.

The sharp one-dimensional theorem is
\[
 \KoDim{1}=\cK.
\]
K\"onig conjectured that the half-space is extremal in every dimension,
namely,
\[
 \KoDim{N}=\cK \qquad (N\geq1).
\]
The direct connection with the continuous form
\cref{eq:GI-continuous} explains why this conjecture bears on $\KG$.
Put
\[
 K_N(\xvec,\yvec):=
 e^{-(\norm{\xvec}^2+\norm{\yvec}^2)/2}\sin\ip{\xvec}{\yvec}.
\]
Apply \cref{eq:GI-continuous} with $a=K_N$. Its sign supremum is
$(\sqrt{2}\pi)^N\KoDim{N}$, while the Hilbert-valued test
$u_N(\xvec)=v_N(\xvec)=\xvec/\norm{\xvec}$ (away from the origin) has
value
\[
 \iint K_N(\xvec,\yvec)\ip{u_N(\xvec)}{v_N(\yvec)}\,\dd\xvec\,\dd\yvec
 = (\sqrt{2}\pi)^N\left(1-\frac1N+O\!\left(N^{-2}\right)\right).
\]
Consequently,
\[
 \KG\KoDim{N}\geq1-\frac1N+O\!\left(N^{-2}\right).
\]
K\"onig's conjecture would therefore give $\KG\geq1/\cK$ as
$N\to\infty$, and Krivine's rounding bound gives the reverse inequality.
Thus the conjecture would make Krivine's constant exact
\cite[Proposition~1.2 and equation~(3.4)]{BMMN}. K\"onig's formulation
was motivated by unpublished computations of Haagerup reported in
\cite[Section~1.1]{BMMN}; compare also Haagerup's classical complex
upper-bound method~\cite{HAAGERUP}.

K\"onig~\cite{KONIG} announced the one-dimensional statement as joint unpublished
work with Tomczak--Jaegermann. It was subsequently proved, together
with its equality cases, by Braverman, Makarychev,
Makarychev, and Naor (BMMN)~\cite[Theorem~1.4]{BMMN};
Appendix~\ref{app:one-dimensional-whole-space} gives a new direct proof.
A breakthrough of BMMN was the surprising planar counterexample showing 
$\KoDim{2}>\cK$~\cite{BMMN}. Their perturbative rounding argument
turned this counterexample into the first strict improvement of
Krivine's bound. They also asked how the maximizers and the associated
schemes behave in higher dimension~\cite[Question~3.2]{BMMN}.

Later work has pursued several distinct directions. Naor and
Regev~\cite{NR} showed that finite-dimensional oblivious Krivine schemes
approximate $\KG$ arbitrarily well, while Krivine~\cite{KRIVINENOTE}
gave a three-dimensional rotating-half-space construction. Quantitative
upper-bound improvements refine BMMN's perturbative Krivine-rounding
approach, while recent lower-bound improvements perturb the Davie--Reeds
construction in cubic Gaussian/Hermite directions; see
\cite{HEILMAN,LISK} and \cite{HEILMANLOWER,JM}, respectively. In particular, Li et al \cite{LISK} constructed high-dimensional affine-radial
partitions showing that the K\"onig constant is greater
than $0.59357$.

\subsection{Main results}\label{subsec:main-results}

Our main result determines the extremal value of the
K\"onig problem, after allowing the dimension to vary. It shows that
increasingly complicated higher-dimensional partitions do not merely
improve on the half-space value; they approach the elementary Fourier
upper bound.

\begin{theorem}\label{thm:main}
There exists an explicit family of measurable functions, indexed by
$m\in\Z_{\geq3}$ and $d\in\Z_{\geq1}$,
\[
 f_{d,m},g_{d,m}:\R^{2m+3d\binom m3}\longrightarrow\{\pm1\},
\]
each of which is odd almost everywhere, such that
\begin{equation}\label{eq:main-limit}
 \lim_{m\to\infty}\lim_{d\to\infty}\BK(f_{d,m},g_{d,m})=1.
\end{equation}
Consequently, $\Ko=1$. Moreover, for every fixed $N$ there exists
$\eta_N>0$ such that
\[
 \abs{\BK(f,g)}\leq1-\eta_N
\]
for all measurable Boolean $f,g:\R^N\to\{\pm1\}$. Thus the supremum is
not attained in finite dimension, and every extremizing sequence has
unbounded dimension.
\end{theorem}

Although we do not determine the exact values of the finite-dimensional
K\"onig constants $\mathfrak K_{\mathrm K,N}$ or characterize their extremizers,
our main theorem shows that the oscillatory Gaussian kernel of the K\"onig
functional cannot determine $K_{\mathrm G}$ through the high-dimensional
testing procedure envisaged by BMMN: both its normalized Boolean norm and
the natural Hilbert-valued test tend to $1$. Thus, if there were infinitely
many dimensions $N$ for which maximizing pairs
$(f_{\max}^{(N)},g_{\max}^{(N)})$ formed alternating Krivine rounding
schemes, this would yield the false upper bound $\KG\leq1$. In this sense,
our result gives a negative answer to the high-dimensional form of the
second part of Question~3.2 in \cite{BMMN}.

\begin{remark}[Monotonicity in the dimension]
The sequence $(\KoDim{N})_{N\geq1}$ is nondecreasing. Indeed, padding
$f$ and $g$ by an unused coordinate leaves $\BK$ unchanged: in
$\sin(\ip{\xvec}{\yvec}+uv)$, the term containing $\sin(uv)$ integrates
to zero, while the $\cos(uv)$ term contributes exactly the additional
normalizing factor $\sqrt2\pi$. Together with the one-dimensional
theorem, BMMN's planar counterexample, and \cref{thm:main}, this gives
\[
 \KoDim{1}=\cK<\KoDim{2}\leq\KoDim{3}\leq\cdots<1,
 \qquad
 \lim_{N\to\infty}\KoDim{N}=1.
\]
\end{remark}

The strict fixed-dimensional inequality in this hierarchy is an elementary
complement rather than the main new ingredient of the theorem.  BMMN
observed in dimension two that the
maximum of K\"onig's bilinear form is attained by a weak-compactness argument
\cite[(3.1) and the paragraph following it]{BMMN}; the same argument works in
every fixed dimension.  Combining attainment with the standard Fourier upper
bound and the obstruction to equality gives the qualitative gap below.  We
have not found this strict fixed-dimensional gap stated explicitly in the
literature, so Appendix~\ref{app:Fourier-gap} records a direct quantitative
proof that also supplies an explicit positive $\eta_N$.

The elementary upper bound has a simple Fourier explanation. After the
Gaussian weights are absorbed into $f$ and $g$, the K\"onig's bilinear form
becomes an $L_2$ pairing through the normalized sine transform.
Cauchy--Schwarz therefore gives $|\BK(f,g)|\leq1$. Equality would force a
sine transform, which is continuous and vanishes at the origin, to agree
up to sign with a Gaussian-weighted Boolean function, whose modulus is
nonzero there. In a fixed dimension this mismatch is uniform on a small
ball around the origin, giving the gap $\eta_N>0$. Thus approaching the
upper bound necessarily requires dimensions tending to infinity. The
short quantitative argument is recorded in
Appendix~\ref{app:Fourier-gap}.

Appendix~\ref{app:two-dimensional-quadrant} also records a sharp scalar
quadrant estimate, obtained by combining the half-line sine estimate from
Appendix~\ref{app:one-dimensional-whole-space} with a cosine factor.

\subsection{Proof sketch}\label{subsec:proof-sketch}

The construction and asymptotic analysis are organized in three steps,
carried out in \cref{sec:halfspace,sec:determinant,sec:parity}. The
proof of \cref{thm:main} is completed in
\cref{sec:main-results-proof}.

\medskip
\noindent\textbf{Step 1: define the candidates and integrate the half-space
layer.}
The construction has two kinds of variables. The vectors
$\Xvec,\Yvec\in\R^{2m}$ are the \emph{half-space variables}; after the
phase variables are fixed, the candidates are ordinary Gaussian
half-spaces. The vectors $\uvec,\vvec\in\R^{3d\binom m3}$ are the
\emph{phase variables}; quadratic functions of them determine unit
normals $\avec(\uvec),\bvec(\vvec)\in\R^{2m}$. Writing $\theta_s$ for
the quadratic phase in channel $s$ (defined explicitly in
\eqref{eq:theta}), these normals are
\begin{align}
 \avec(\uvec)&:=\frac1{\sqrt m}
 \bigl(\cos\theta_1(\uvec),\sin\theta_1(\uvec),\ldots,
       \cos\theta_m(\uvec),\sin\theta_m(\uvec)\bigr),
 \label{eq:intro-a}\\
 \bvec(\vvec)&:=\frac1{\sqrt m}
 \bigl(\cos\theta_1(\vvec),-\sin\theta_1(\vvec),\ldots,
       \cos\theta_m(\vvec),-\sin\theta_m(\vvec)\bigr).
 \label{eq:intro-b}
\end{align}
Thus
\[
 \ip{\avec(\uvec)}{\bvec(\vvec)}
 =\frac1m\sum_{s=1}^m
 \cos\bigl(\theta_s(\uvec)+\theta_s(\vvec)\bigr).
\]
We set
\begin{equation}\label{eq:intro-functions}
 f_{d,m}(\Xvec,\uvec)=\sgn\ip{\Xvec}{\avec(\uvec)},\qquad g_{d,m}(\Yvec,\vvec)=\sgn\ip{\Yvec}{\bvec(\vvec)}.
\end{equation}
The key principle is the oscillatory half-space identity: for unit
vectors $\avec,\bvec\in\R^k$,
\begin{equation}\label{eq:intro-halfspace}
 \frac1{(\sqrt{2}\pi)^k}\iint \sgn\ip{\Xvec}{\avec}\,\sgn\ip{\Yvec}{\bvec}e^{-(\norm{\Xvec}^2+\norm{\Yvec}^2)/2}\sin\ip{\Xvec}{\Yvec}\,\dd\Xvec\,\dd\Yvec=\frac2\pi\arsinh\ip{\avec}{\bvec}.
\end{equation}
It is the analytic continuation to the imaginary correlation
$\iu:=\sqrt{-1}$ of the usual
one-dimensional Gaussian arcsine law \eqref{eq:hyperplane-correlation}.
Integrating first in $\Xvec,\Yvec$ therefore removes all Boolean
discontinuities and leaves a smooth integral in $\uvec,\vvec$. The
normal inner product is an average of $m$ cosines. Expanding those
cosines and the absolutely convergent series for $\arsinh$ gives
\begin{equation}\label{eq:intro-coeff}
 \BK(f_{d,m},g_{d,m})=\frac2\pi\sum_{n\geq0}(-1)^nc_nT_{d,m,n},\qquad c_n>0,\qquad \frac2\pi\sum_{n\geq0}c_n=1.
\end{equation}
More explicitly, write $[m]:=\{1,\ldots,m\}$. Let
$\tau_1,\ldots,\tau_{2n+1}$ be independent and uniform in $[m]$, and
let $\sigma_1,\ldots,\sigma_{2n+1}$ be independent
uniform signs, independent also of the $\tau_t$'s. Write
\[
 \tauvec:=(\tau_1,\ldots,\tau_{2n+1}),\qquad
 \sigmavec:=(\sigma_1,\ldots,\sigma_{2n+1}),
\]
and define the random frequency vector
\[
 \muvec=(\mu_1,\ldots,\mu_m),
 \qquad
 \mu_s:=\sum_{t=1}^{2n+1}\sigma_t\mathbf 1_{\{\tau_t=s\}},\qquad
 \eta_s(\uvec,\vvec):=\theta_s(\uvec)+\theta_s(\vvec).
\]
With $N_2=3d\binom m3$, the coefficient in
\eqref{eq:intro-coeff} is
\begin{equation}\label{eq:intro-T-explicit}
 \begin{split}
 T_{d,m,n}:={}&\E_{\tauvec,\sigmavec}\Bigg[
 \frac1{(\sqrt{2}\pi)^{N_2}}
 \iint_{\R^{N_2}\times\R^{N_2}}
 \exp\!\left(-\frac{\norm{\uvec}^2+\norm{\vvec}^2}{2}
 +\iu\ip{\uvec}{\vvec}
 +\iu\sum_{s=1}^m\mu_s\eta_s(\uvec,\vvec)\right)
 \,\dd\uvec\,\dd\vvec\Bigg].
 \end{split}
\end{equation}
Thus the problem is reduced to understanding one phase integral
$T_{d,m,n}$ at each odd degree $2n+1$.

The exponent in \eqref{eq:intro-T-explicit} is quadratic in
$(\uvec,\vvec)$, so the remaining integral is Gaussian. In general, if
$M$ is a complex symmetric $D\times D$ matrix with positive-definite
Hermitian real part, then the branch-free squared Gaussian identity is
\begin{equation}\label{eq:intro-Gaussian-formula}
 \left(\frac1{(2\pi)^{D/2}}\int_{\R^D}
 e^{-\zvec^{\mathsf T}M\zvec/2}\,\dd\zvec\right)^2
 =\det(M)^{-1}.
\end{equation}
After decomposing the phase coordinates into independent
three-dimensional blocks, \eqref{eq:intro-Gaussian-formula} determines
the square of each block integral by a determinant. This is the reason
for the determinant computation in the next step.

\medskip
\noindent\textbf{Step 2: evaluate the phase integral by Gaussian
determinants.}
The phase coordinates are grouped into independent three-dimensional
blocks, indexed by $1\le j<k<\ell\le m$ and $1\le r\le d$. The vector
$\muvec=(\mu_1,\ldots,\mu_m)$ in \eqref{eq:intro-T-explicit} is the net
frequency vector produced by the cosine expansion: $\mu_s$ is the signed
number of times channel $s$ was selected. Consequently, a block indexed
by $(j,k,\ell)$ sees the three frequencies
$(\mu_j,\mu_k,\mu_\ell)$. More generally, a block with frequencies
$(a,b,c)$ is governed by the three-cycle matrix
\begin{equation}\label{eq:intro-A}
 A(a,b,c)=\begin{pmatrix}0&c&b\\ c&0&a\\ b&a&0\end{pmatrix}.
\end{equation}
Denote the normalized Gaussian block integral by
$\phi_\eps(a,b,c)$. The squared Gaussian formula gives
\begin{equation}\label{eq:intro-phi}
 \phi_\eps(a,b,c)^2
 =\det\!\left(\Id-\iu\eps A(a,b,c)-\frac{\eps^2}{2}A(a,b,c)^2\right)^{-1}.
\end{equation}
Thus the determinant computation does not require choosing a branch of
the complex square root. When $\phi_\eps$ itself is recovered near
$\eps=0$, its integral definition and the normalization $\phi_0=1$
select the unique local branch.
Applying the Gaussian formula block by block to the expression for
$T_{d,m,n}$ in \eqref{eq:intro-T-explicit} gives the block
representation
\begin{equation}\label{eq:intro-T-product}
 T_{d,m,n}
 =\E_{\tauvec,\sigmavec}
 \prod_{1\leq j<k<\ell\leq m}
 \phi_{\eps_d}(\mu_j,\mu_k,\mu_\ell)^d.
\end{equation}
Thus the analysis of $T_{d,m,n}$ reduces exactly to understanding one
$3\times3$ determinant for each channel triple.
The determinant in \eqref{eq:intro-phi} equals
$1-\iu\eps^3abc+O(\eps^4)$. Since the block integral is continuous and
equals $1$ at $\eps=0$, the squared identity gives
$\phi_\eps(a,b,c)=1+\iu\eps^3abc/2+O(\eps^4)$. With
$\eps_d=(2\pi/d)^{1/3}$, the $d$-fold product therefore satisfies
\begin{equation}\label{eq:intro-one-bit}
 \phi_{\eps_d}(\mu_j,\mu_k,\mu_\ell)^d
 \xrightarrow[d\to\infty]{}(-1)^{\mu_j\mu_k\mu_\ell}.
\end{equation}
The full phase integral factors over all triples and converges to the
cubic character
\begin{equation}\label{eq:intro-cubic-character}
 (-1)^{e_3(\muvec)},\qquad e_3(\muvec):=\sum_{j<k<\ell}\mu_j\mu_k\mu_\ell.
\end{equation}

\medskip
\noindent\textbf{Step 3: identify cubic parity with degree parity.}
At degree $2n+1$ in \eqref{eq:intro-coeff}, the cosine expansion chooses
$2n+1$ channel indices. For fixed $n$, these indices are distinct with
probability tending to one as $m\to\infty$. On this collision-free
event the frequency vector has exactly $2n+1$ nonzero coordinates, all
in $\{\pm1\}$. Hence
\begin{equation}\label{eq:intro-parity}
 e_3(\muvec)\equiv\binom{2n+1}{3}\equiv n\pmod2.
\end{equation}
Consequently,
\begin{equation}\label{eq:intro-T-limit}
 \lim_{m\to\infty}\lim_{d\to\infty}T_{d,m,n}=(-1)^n.
\end{equation}
This sign cancels the alternating sign of the $\arsinh$ coefficient in
every fixed odd degree. Absolute summability permits dominated
convergence in \eqref{eq:intro-coeff}, and the remaining positive
coefficients sum to $1$.

\subsection{Application to improving the Grothendieck constant}
\label{subsec:grothendieck-application}

The same construction has a qualitative consequence for the
Grothendieck constant.

\begin{corollary}\label{cor:groth}
There is a randomized Krivine scheme such that
\[
 \KG<\frac\pi{2\log(1+\sqrt2)}.
\]
\end{corollary}

\begin{remark}
By \cref{thm:main}, some finite pair $(f_{d,m},g_{d,m})$ satisfies
$\BK(f_{d,m},g_{d,m})>\cK$. The perturbative mixing argument of
BMMN~\cite[Section~5 and Remark~5.6]{BMMN}, in the
arbitrary-dimensional form of \cite[Theorem~6.2]{LISK}, yields the
corollary.
\end{remark}

The ambient dimension is denoted by $N$; $m$ is the number of phase
channels, $d$ is the number of copies of each three-cycle block, and $n$
indexes the odd power-series degree $2n+1$.

\subsection*{Comments on use of AI tools}
The main results of this paper were discovered through dialogues between the authors and ChatGPT 5.5 Pro. The authors also acknowlege the use of ChatGPT 5.6 Pro.

\subsection*{Acknowledgments}
X. X. is grateful to Paata Ivanisvili for suggesting related problems and for helpful discussions, and to Roman Vershynin for an enlightening lecture on the Grothendieck inequality. H.Z. is supported by NSF DMS-2453408. 

\section{The candidates and the half-space reduction}
\label{sec:halfspace}

We begin with the two-layer construction announced above. Fix
integers $m\geq3$ and $d\geq1$, and set
\begin{equation}\label{eq:dimensions}
 N_1:=2m,
 \qquad
 N_2:=3d\binom m3,
 \qquad
 N:=N_1+N_2,
 \qquad
 \eps_d:=\left(\frac{2\pi}{d}\right)^{1/3}.
\end{equation}
We write
\[
 (\Xvec,\uvec)\in\R^{N_1}\times\R^{N_2},
 \qquad
 (\Yvec,\vvec)\in\R^{N_1}\times\R^{N_2}.
\]
The variables $\Xvec,\Yvec$ will be integrated first
to provide a half-space envelope. The variables $\uvec,\vvec$ carry the
quadratic phases that encode the desired Fourier signs.

Regard the phase variables as column block vectors indexed by
$1\leq j<k<\ell\leq m$ and $1\leq r\leq d$:
\[
 \begin{aligned}
 \uvec&=\bigl(\uvec^{(r)}_{jk\ell}\bigr)_{j<k<\ell,\,r},
 &\uvec^{(r)}_{jk\ell}
 &=\bigl(\alpha^{(r)}_{jk\ell},\beta^{(r)}_{jk\ell},
          \gamma^{(r)}_{jk\ell}\bigr)^{\mathsf T},\\
 \vvec&=\bigl(\vvec^{(r)}_{jk\ell}\bigr)_{j<k<\ell,\,r},
 &\vvec^{(r)}_{jk\ell}
 &=\bigl(\widetilde\alpha^{(r)}_{jk\ell},
          \widetilde\beta^{(r)}_{jk\ell},
          \widetilde\gamma^{(r)}_{jk\ell}\bigr)^{\mathsf T}.
 \end{aligned}
\]
Thus $\uvec,\vvec\in(\R^3)^{d\binom m3}\cong\R^{N_2}$.
For $s\in[m]$, define the quadratic phase
\begin{equation}\label{eq:theta}
 \theta_s(\uvec):=\eps_d\sum_{r=1}^d\left(
 \sum_{s<k<\ell}\beta^{(r)}_{sk\ell}\gamma^{(r)}_{sk\ell}
 +\sum_{j<s<\ell}\alpha^{(r)}_{js\ell}\gamma^{(r)}_{js\ell}
 +\sum_{j<k<s}\alpha^{(r)}_{jks}\beta^{(r)}_{jks}\right),
\end{equation}
and define $\theta_s(\vvec)$ by the same formula with all phase
coordinates replaced by their tilded counterparts. Here $s$ is fixed:
the first sum is over $k,\ell$ subject to $s<k<\ell$, the second is over
$j,\ell$ subject to $j<s<\ell$, and the third is over $j,k$ subject to
$j<k<s$. Thus \eqref{eq:theta} puts the product of the
other two coordinates into each phase: $\beta\gamma$ in $\theta_j$,
$\alpha\gamma$ in $\theta_k$, and $\alpha\beta$ in $\theta_\ell$.
After frequency weighting, this gives
$\eps_d(\mu_j\beta\gamma+\mu_k\alpha\gamma+\mu_\ell\alpha\beta)$, the
quadratic form generated by $A(\mu_j,\mu_k,\mu_\ell)$ up to a scalar.

Define two phase-dependent unit vectors in $\R^{2m}$ by
\begin{align}
 \avec(\uvec)&:=\frac1{\sqrt m}
 \bigl(\cos\theta_1(\uvec),\sin\theta_1(\uvec),\ldots,
       \cos\theta_m(\uvec),\sin\theta_m(\uvec)\bigr),
 \label{eq:a}\\
 \bvec(\vvec)&:=\frac1{\sqrt m}
 \bigl(\cos\theta_1(\vvec),-\sin\theta_1(\vvec),\ldots,
       \cos\theta_m(\vvec),-\sin\theta_m(\vvec)\bigr).
 \label{eq:b}
\end{align}
The signs in \eqref{eq:b} are chosen so that
\begin{equation}\label{eq:inner-ab}
 \ip{\avec(\uvec)}{\bvec(\vvec)}
 =\frac1m\sum_{s=1}^m
 \cos\bigl(\theta_s(\uvec)+\theta_s(\vvec)\bigr).
\end{equation}
Finally, set
\begin{equation}\label{eq:functions}
 f_{d,m}(\Xvec,\uvec):=\sgn\ip{\Xvec}{\avec(\uvec)},
 \qquad
 g_{d,m}(\Yvec,\vvec):=\sgn\ip{\Yvec}{\bvec(\vvec)}.
\end{equation}
Because every $\theta_s$ is quadratic,
$\avec(-\uvec)=\avec(\uvec)$ and $\bvec(-\vvec)=\bvec(\vvec)$. Hence
$f_{d,m}$ and $g_{d,m}$ are odd almost everywhere on $\R^N$.

The half-space variables can now be removed exactly using the following
lemma. It is a standard consequence of the Gaussian arcsine
law and analytic continuation. It can also be obtained by specializing
BMMN's correlation formula \cite[(2.2), (2.5)]{BMMN}. We include a direct
proof to keep track of our normalization.

\begin{lemma}[Oscillatory half-space identity]\label{lem:halfspace}
For unit vectors $\avec,\bvec\in\R^k$,
\begin{equation}\label{eq:halfspace-identity}
 \frac1{(\sqrt{2}\pi)^k}\iint_{\R^k\times\R^k}
 \sgn\ip{\Xvec}{\avec}\,\sgn\ip{\Yvec}{\bvec}
 e^{-\frac{1}{2}(\norm{\Xvec}^2+\norm{\Yvec}^2)}
 \sin\ip{\Xvec}{\Yvec}\,\dd\Xvec\,\dd\Yvec
 =\frac2\pi\arsinh\ip{\avec}{\bvec}.
\end{equation}
\end{lemma}

\begin{proof}
Put $\rho=\ip{\avec}{\bvec}$ and, for $|\operatorname{Re}t|<1$, define
\[
 F_k(t):=\frac1{(\sqrt{2}\pi)^k}\iint_{\R^k\times\R^k}
 \sgn\ip{\Xvec}{\avec}\,\sgn\ip{\Yvec}{\bvec}
 e^{-\frac{1}{2}(\norm{\Xvec}^2+\norm{\Yvec}^2)+t\ip{\Xvec}{\Yvec}}
 \,\dd\Xvec\,\dd\Yvec.
\]
On every closed substrip $|\operatorname{Re}t|\leq1-\delta$, the
integrand is dominated by an integrable Gaussian, since
\[
 \operatorname{Re}\bigl(t\ip{\Xvec}{\Yvec}\bigr)
 \leq\frac{|\operatorname{Re}t|}{2}
       (\norm{\Xvec}^2+\norm{\Yvec}^2).
\]
Thus $F_k$ is holomorphic in the strip. To compute it on the real
interval, fix $t\in(-1,1)$ and let $\Gvec_1,\Gvec_2$ be independent
standard Gaussian vectors in $\R^k$. A direct change of variables gives
\[
 F_k(t)=\left(\frac2{1-t^2}\right)^{k/2}
 \E\!\left[
 \sgn\ip{\Gvec_1}{\avec}\,
 \sgn\ip{t\Gvec_1+\sqrt{1-t^2}\,\Gvec_2}{\bvec}
 \right].
\]
In \eqref{eq:hyperplane-correlation}, take the standard Gaussian vector
$(\Gvec_1,\Gvec_2)\in\R^{2k}$ and the two unit vectors
\[
 (\avec,0),
 \qquad
 \bigl(t\bvec,\sqrt{1-t^2}\,\bvec\bigr),
\]
whose inner product is $t\ip{\avec}{\bvec}=t\rho$. Hence
\eqref{eq:hyperplane-correlation} gives
\[
 F_k(t)=\left(\frac2{1-t^2}\right)^{k/2}
        \frac2\pi\arcsin(t\rho).
\]
We now return to complex $t$. Choose the branch of
$(1-t^2)^{-k/2}$ that equals $1$ at $t=0$ and the
branch of $\arcsin$ satisfying $\arcsin(0)=0$ and $\arcsin'(0)=1$.
Both are holomorphic in the strip, so the identity theorem extends the
formula from the real interval throughout the strip. At $t=\iu$, the prefactor is $1$ and
$\arcsin(\iu\rho)=\iu\arsinh\rho$. The real part of $F_k(\iu)$ vanishes
under $\Xvec\mapsto-\Xvec$, while its imaginary part is the left-hand
side of \eqref{eq:halfspace-identity}. This proves the identity.
\end{proof}

Set
\begin{equation}\label{eq:eta-Z}
 \eta_s(\uvec,\vvec):=\theta_s(\uvec)+\theta_s(\vvec),
 \qquad
 Z(\uvec,\vvec):=\frac1m\sum_{s=1}^m\cos\eta_s(\uvec,\vvec).
\end{equation}
In the full K\"onig kernel,
\[
 \sin\bigl(\ip{\Xvec}{\Yvec}+\ip{\uvec}{\vvec}\bigr)
 =\sin\ip{\Xvec}{\Yvec}\cos\ip{\uvec}{\vvec}
  +\cos\ip{\Xvec}{\Yvec}\sin\ip{\uvec}{\vvec}.
\]
The second term has zero $\Xvec,\Yvec$ integral: under
$\Xvec\mapsto-\Xvec$, the factor
$\sgn\ip{\Xvec}{\avec(\uvec)}$ changes sign, whereas
$\sgn\ip{\Yvec}{\bvec(\vvec)}$ and $\cos\ip{\Xvec}{\Yvec}$ do not.
Applying \cref{lem:halfspace} to the first term and using
\eqref{eq:inner-ab} gives
\begin{equation}\label{eq:after-halfspace}
 \BK(f_{d,m},g_{d,m})
 =\frac2\pi\frac1{(\sqrt{2}\pi)^{N_2}}
 \iint e^{-(\norm{\uvec}^2+\norm{\vvec}^2)/2}
 \cos\ip{\uvec}{\vvec}\,\arsinh Z(\uvec,\vvec)
 \,\dd\uvec\,\dd\vvec.
\end{equation}
Now all Boolean discontinuities have
been integrated out, and the remaining integrand is built from smooth
quadratic phases.

Write
\begin{equation}\label{eq:arsinh-series}
 \arsinh z=\sum_{n=0}^\infty(-1)^nc_nz^{2n+1},
 \qquad
 c_n:=\frac{\binom{2n}{n}}{4^n(2n+1)}.
\end{equation}
The numbers $c_n$ are also the positive Taylor coefficients of
$\arcsin z$. Hence
\[
 \sum_{n\geq0}c_n=\arcsin(1)=\frac\pi2.
\]
Fix $n\geq0$. Let $\tau_1,\ldots,\tau_{2n+1}$ be independent and uniform
in $[m]$, and let $\sigma_1,\ldots,\sigma_{2n+1}$ be independent, uniformly
random signs, also independent of the $\tau_t$'s. Write
\[
 \tauvec:=(\tau_1,\ldots,\tau_{2n+1}),\qquad
 \sigmavec:=(\sigma_1,\ldots,\sigma_{2n+1}).
\]
The vector $\tauvec$ records which channel term
$\eta_s(\uvec,\vvec)$ is selected from each factor of $Z(\uvec,\vvec)$:
\[
Z(\uvec,\vvec)=\frac1m\sum_{s=1}^m\cos\eta_s(\uvec,\vvec)
=\sum_{s=1}^m\frac{e^{\iu\eta_s(\uvec,\vvec)}+e^{-\iu\eta_s(\uvec,\vvec)}}{2m}
\]
while $\sigmavec$ records which of the two exponentials is selected.
Define the random frequency vector
$\muvec=\muvec(\tauvec,\sigmavec)\in\Z^m$ by
\begin{equation}\label{eq:mu}
 \begin{aligned}
 \mu_s
 &:=\sum_{t=1}^{2n+1}\sigma_t\mathbf 1_{\{\tau_t=s\}},
 \qquad s\in[m].
 \end{aligned}
\end{equation}
For example, if $\tauvec=(2,5,2)$ and
$\sigmavec=(1,-1,1)$, then the corresponding frequency is
$2\bm e_2-\bm e_5$, where $\bm e_s$ denotes the $s$th standard
coordinate vector.

For $\muvec=(\mu_1,\ldots,\mu_m)\in\Z^m$,
define the phase integral
\begin{equation}\label{eq:I-def}
 \mathcal I_{d,m}(\muvec):=\frac1{(\sqrt{2}\pi)^{N_2}}\iint_{\R^{N_2}\times\R^{N_2}}
 \exp\!\left(-\frac{\norm{\uvec}^2+\norm{\vvec}^2}{2}
 +\iu\ip{\uvec}{\vvec}+\iu\sum_{s=1}^m\mu_s\eta_s(\uvec,\vvec)\right)
 \,\dd\uvec\,\dd\vvec.
\end{equation}
We then set
\begin{equation}\label{eq:T-def}
 T_{d,m,n}:=\E_{\tauvec,\sigmavec}
 \mathcal I_{d,m}\bigl(\muvec(\tauvec,\sigmavec)\bigr).
\end{equation}

\begin{proposition}[Coefficient formula]\label{prop:coefficient-formula}
For every $d,m$,
\begin{equation}\label{eq:coefficient-formula}
 \BK(f_{d,m},g_{d,m})
 =\frac2\pi\sum_{n=0}^\infty(-1)^nc_nT_{d,m,n}.
\end{equation}
\end{proposition}

\begin{proof}
Since $\sum_{n\geq0}c_n=\pi/2$, the series in
\eqref{eq:arsinh-series} converges absolutely and uniformly for
$z\in[-1,1]$. Moreover, $Z(\uvec,\vvec)$ is an average of cosines, so
$Z(\uvec,\vvec)\in[-1,1]$. We may therefore insert the series into
\eqref{eq:after-halfspace} and integrate term by term. For each fixed
$n$, the identity $\cos x=(e^{\iu x}+e^{-\iu x})/2$ and
\eqref{eq:mu} give
\begin{align}\label{eq:random-cosine-expansion}
 Z(\uvec,\vvec)^{2n+1}
 &=\frac1{2^{2n+1}m^{2n+1}}
 \sum_{\tauvec\in[m]^{2n+1}}
 \sum_{\sigmavec\in\{\pm1\}^{2n+1}}
 \exp\!\left(\iu\sum_{s=1}^m
 \mu_s\eta_s(\uvec,\vvec)\right)\notag\\
 &=\E_{\tauvec,\sigmavec}
 \exp\!\left(\iu\sum_{s=1}^m
 \mu_s\eta_s(\uvec,\vvec)\right).
\end{align}
Now write
$\cos\ip{\uvec}{\vvec}$ as the average of
$e^{\iu\ip{\uvec}{\vvec}}$ and
$e^{-\iu\ip{\uvec}{\vvec}}$. The two resulting integrals agree under
$\vvec\mapsto-\vvec$, because every $\theta_s(\vvec)$ is even. Thus the
integral associated with a fixed pair $(\tauvec,\sigmavec)$ is exactly
\eqref{eq:I-def}, and averaging gives \eqref{eq:coefficient-formula}.
\end{proof}

We will prove in \cref{prop:block-factorization} that
$|T_{d,m,n}|\leq1$. This bound, together with
$\frac2\pi\sum_{n\geq0}c_n=1$, shows at once from
\eqref{eq:coefficient-formula} that
$|\BK(f_{d,m},g_{d,m})|\leq1$. Thus it remains to prove that, for each
fixed $n$, the signed coefficient approaches its extremal value:
\[
 \lim_{m\to\infty}\lim_{d\to\infty}T_{d,m,n}=(-1)^n.
 \]

\section{Gaussian determinants and the phase integral}
\label{sec:determinant}

We now compute $\mathcal I_{d,m}(\muvec)$ by reducing it to the analysis
of $3\times3$ matrices. For $a,b,c\in\R$, put
\begin{equation}\label{eq:Aabc}
 A(a,b,c):=\begin{pmatrix}0&c&b\\ c&0&a\\ b&a&0\end{pmatrix},\qquad p:=a^2+b^2+c^2,\qquad q:=abc.
\end{equation}
For $\xvec=(\alpha,\beta,\gamma)^{\mathsf T}$, one has 
\[
 \frac{1}{2}\xvec^{\mathsf T}A(a,b,c)\xvec
 =a\beta\gamma+b\alpha\gamma+c\alpha\beta.
\]
Define
\begin{equation}\label{eq:Delta-phi}
 \Delta_\eps(a,b,c):=\det\!\left(\Id-\iu\eps A(a,b,c)-\frac{\eps^2}{2}A(a,b,c)^2\right).
\end{equation}
A direct expansion gives
\[
 \det(t\Id-A)=t^3-pt-2q,
\]
and
\[
 \Id-\iu\eps A-\frac{\eps^2}{2}A^2
 =\frac12\bigl((1-\iu)\Id-\iu\eps A\bigr)
          \bigl((1+\iu)\Id-\iu\eps A\bigr).
\]
Since
\[
 \det(\alpha\Id-\iu\eps A)
 =\alpha^3+\alpha p\eps^2+2\iu q\eps^3,
\]
substitution of $\alpha=1\pm\iu$ gives
\begin{equation}\label{eq:exact-determinant}
 \Delta_\eps(a,b,c)=1-\iu q\eps^3+\frac{p^2}{4}\eps^4+\frac{\iu pq}{2}\eps^5-\frac{q^2}{2}\eps^6.
\end{equation}
In particular, the linear and quadratic terms vanish. If $\lambda$ is
an eigenvalue of $A$, then
\begin{equation}\label{eq:contractive-factor}
 \left|1-\iu\eps\lambda-\frac{\eps^2\lambda^2}{2}\right|^2
 =1+\frac{\eps^4\lambda^4}{4}.
\end{equation}
Consequently $\Delta_\eps(a,b,c)$ never vanishes for real $\eps$ and
$|\Delta_\eps(a,b,c)|\geq1$.

\begin{proposition}[Gaussian determinant formula]\label{prop:determinant}
For $a,b,c,\eps\in\R$, with $A=A(a,b,c)$, let
\begin{equation}\label{eq:determinant-formula}
 \phi_\eps(a,b,c):=\frac1{(\sqrt{2}\pi)^3}\iint_{\R^3\times\R^3}\exp\!\left(-\frac{\norm{\xvec}^2+\norm{\yvec}^2}{2}+\iu\ip{\xvec}{\yvec}+\frac{\iu\eps}{2}\xvec^{\mathsf T}A\xvec+\frac{\iu\eps}{2}\yvec^{\mathsf T}A\yvec\right)\dd\xvec\,\dd\yvec.
\end{equation}
Then
\begin{equation}\label{eq:determinant-evaluation}
 \phi_\eps(a,b,c)=\Delta_\eps(a,b,c)^{-1/2},
\end{equation}
where the inverse square root is the continuous branch for real $\eps$
that equals $1$ at $\eps=0$. In particular, $\phi_0(a,b,c)=1$ and
$|\phi_\eps(a,b,c)|\leq1$. Moreover,
\begin{equation}\label{eq:conjugation}
 \phi_\eps(-a,-b,-c)=\overline{\phi_\eps(a,b,c)}.
\end{equation}
Finally, there are absolute constants $c_0,C>0$ such that, whenever
$L\geq1$, $|a|,|b|,|c|\leq L$, and $|\eps|L\leq c_0$,
\begin{equation}\label{eq:local-three-cycle}
 \phi_\eps(a,b,c)=1+\frac{\iu abc\eps^3}{2}+O(L^4\eps^4),\qquad
 \left|e^{-\iu abc\eps^3/2}\phi_\eps(a,b,c)-1\right|
 \leq CL^4|\eps|^4.
\end{equation}
\end{proposition}

\begin{proof}
For the integral in \eqref{eq:determinant-formula}, put
$\zvec=(\xvec,\yvec)\in\R^6$. Its exponent is
$-\zvec^{\mathsf T}\mathcal M\zvec/2$, where
\[
 \mathcal M=\begin{pmatrix}\Id-\iu\eps A&-\iu\Id\\-\iu\Id&\Id-\iu\eps A\end{pmatrix}.
\]
The matrix $\mathcal M$ is complex symmetric, and its real part is the
identity. Therefore the standard complex Gaussian identity
\cite[Appendix~A, Theorem~1]{FOLLAND}, evaluated at zero and rescaled to our
normalization, gives
\[
 \frac1{(2\pi)^3}\int_{\R^6}
 e^{-\zvec^{\mathsf T}\mathcal M\zvec/2}\,\dd\zvec
 =(\det\mathcal M)^{-1/2}.
\]
Here the inverse square root is the holomorphic branch on the complex
symmetric matrices with positive-definite real part obtained by
continuation from real positive-definite matrices, where it is positive.
Since $\phi_\eps(a,b,c)$ is $2^{3/2}$ times the normalized integral
above, it remains only to compute $\det\mathcal M$. The orthogonal matrix
\[
 \mathcal U:=\frac1{\sqrt2}
 \begin{pmatrix}\Id&\Id\\ \Id&-\Id\end{pmatrix}
\]
gives
\[
 \mathcal U^{\mathsf T}\mathcal M\mathcal U=\begin{pmatrix}\Id-\iu\eps A-\iu\Id&0\\0&\Id-\iu\eps A+\iu\Id\end{pmatrix}.
\]
The two blocks commute and their product is
$2(\Id-\iu\eps A-\eps^2A^2/2)$, so
$\det\mathcal M=2^3\Delta_\eps(a,b,c)$. Therefore
\[
 \phi_\eps(a,b,c)=\Delta_\eps(a,b,c)^{-1/2},
\]
with the branch inherited from the preceding identity.
At $\eps=0$, integrating first in $\yvec$ gives
$(2\pi)^{3/2}e^{-\norm{\xvec}^2/2}$; the remaining integral equals
$\pi^{3/2}$, and the normalization in
\eqref{eq:determinant-formula} gives $\phi_0(a,b,c)=1$.

The bound $|\phi_\eps(a,b,c)|\leq1$ follows from
\eqref{eq:determinant-evaluation} and \eqref{eq:contractive-factor}.
Identity \eqref{eq:conjugation} follows directly from
\eqref{eq:determinant-formula}, followed by the change of variables
$\yvec\mapsto-\yvec$.

We finally prove the local expansion. Under the stated assumptions,
\eqref{eq:exact-determinant} gives
\[
 w:=\Delta_\eps(a,b,c)-1
 =-\iu q\eps^3+O(L^4\eps^4).
\]
Since $L|\eps|\leq c_0$, this implies
\[
 |w|^2\leq CL^6|\eps|^6
 \leq Cc_0^2L^4|\eps|^4.
\]
After decreasing $c_0$ if necessary, $|w|<1/2$. On this neighborhood of
$1$, let $\varphi(z)=z^{-1/2}$ denote the analytic branch satisfying
$\varphi(1)=1$. The other branch is $-\varphi$ and equals $-1$ at $1$, so it is excluded by
$\phi_0(a,b,c)=1$. Thus \eqref{eq:determinant-evaluation} and the Taylor
expansion
\[
 \varphi(1+w)=1-\frac12w+O(w^2)
\]
give
\[
 \phi_\eps(a,b,c)
 =1+\frac{\iu q\eps^3}{2}+O(L^4\eps^4).
\]
Finally,
$e^{-\iu q\eps^3/2}=1-\iu q\eps^3/2+O(q^2\eps^6)$ and
$q^2\eps^6=O(L^4\eps^4)$ in the same range, proving
\eqref{eq:local-three-cycle}.
\end{proof}

\begin{corollary}[One triple produces one parity bit]\label{cor:one-triple}
There is an absolute constant $C$ such that, for every $L\geq1$, all
integers $a,b,c$ satisfying $|a|,|b|,|c|\leq L$, and every
$d\geq CL^3$,
\begin{equation}\label{eq:one-triple}
 \left|\phi_{\eps_d}(a,b,c)^d-(-1)^{abc}\right|\leq CL^4d^{-1/3}.
\end{equation}
\end{corollary}

\begin{proof}
After enlarging $C$ if necessary, the hypothesis $d\geq CL^3$ ensures
\[
 |\eps_d|L\leq c_0,
\]
where $c_0$ is the constant above. Set
\[
 R_d:=e^{-\iu abc\eps_d^3/2}\phi_{\eps_d}(a,b,c).
\]
By \eqref{eq:local-three-cycle},
\[
 |R_d-1|\leq CL^4d^{-4/3}.
\]
Moreover, contractivity gives $|R_d|\leq1$. Therefore
\[
 |R_d^d-1|
 \leq |R_d-1|\sum_{r=0}^{d-1}|R_d|^r
 \leq d|R_d-1|
 \leq CL^4d^{-1/3}.
\]
Finally, $d\eps_d^3=2\pi$, and hence
\[
 \phi_{\eps_d}(a,b,c)^d
 =e^{\iu\pi abc}R_d^d
 =(-1)^{abc}R_d^d.
\]
It follows that
\[
 \left|\phi_{\eps_d}(a,b,c)^d-(-1)^{abc}\right|
 =|R_d^d-1|
 \leq CL^4d^{-1/3}.
\]
\end{proof}

For $\muvec\in\Z^m$ and $j<k<\ell$, set
\begin{equation}\label{eq:Ajkl}
 A_{jk\ell}(\muvec):=A(\mu_j,\mu_k,\mu_\ell)=\begin{pmatrix}0&\mu_\ell&\mu_k\\\mu_\ell&0&\mu_j\\\mu_k&\mu_j&0\end{pmatrix},
\end{equation}
and define
\begin{equation}\label{eq:Phi-def}
 \Phi_{d,m}(\muvec):=\prod_{1\leq j<k<\ell\leq m}\phi_{\eps_d}(\mu_j,\mu_k,\mu_\ell)^d.
\end{equation}
Here and below, the exponent $d$ denotes an ordinary integer power,
that is, repeated multiplication of the block integral.

\begin{proposition}[Block factorization]\label{prop:block-factorization}
For every $\muvec\in\Z^m$,
\begin{equation}\label{eq:block-factorization}
 \mathcal I_{d,m}(\muvec)=\Phi_{d,m}(\muvec).
\end{equation}
Consequently $|\mathcal I_{d,m}(\muvec)|\leq1$ and
$|T_{d,m,n}|\leq1$.
\end{proposition}

\begin{proof}
For one block $\uvec^{(r)}_{jk\ell}=(\alpha,\beta,\gamma)^{\mathsf T}=(\alpha^{(r)}_{jk\ell},\beta^{(r)}_{jk\ell},\gamma^{(r)}_{jk\ell})^{\mathsf T}$, the phases
$\theta_j,\theta_k,\theta_\ell$ receive respectively
$\eps_d\beta\gamma$, $\eps_d\alpha\gamma$, and
$\eps_d\alpha\beta$. Indeed, this block occurs in the first sum of
$\theta_j$, the second sum of $\theta_k$, and the third sum of
$\theta_\ell$. Thus its contribution to
$\sum_{s=1}^m\mu_s\theta_s(\uvec)$ is
\[
 \eps_d\bigl(\mu_j\beta\gamma
 +\mu_k\alpha\gamma+\mu_\ell\alpha\beta\bigr)
 =\frac{\eps_d}{2}
 (\alpha,\beta,\gamma)A_{jk\ell}(\muvec)
 (\alpha,\beta,\gamma)^{\mathsf T}.
\]
Reindexing all the phase sums by triples therefore gives
\begin{align}
 \sum_{s=1}^m\mu_s\theta_s(\uvec)
 &=\eps_d\sum_{j<k<\ell}\sum_{r=1}^d
 \bigl(\mu_j\beta^{(r)}_{jk\ell}\gamma^{(r)}_{jk\ell}
      +\mu_k\alpha^{(r)}_{jk\ell}\gamma^{(r)}_{jk\ell}
      +\mu_\ell\alpha^{(r)}_{jk\ell}\beta^{(r)}_{jk\ell}\bigr)\notag\\
 &=\frac{\eps_d}{2}\sum_{j<k<\ell}\sum_{r=1}^d
 (\uvec^{(r)}_{jk\ell})^{\mathsf T}
 A_{jk\ell}(\muvec)\uvec^{(r)}_{jk\ell}.
 \label{eq:block-phase}
\end{align}
The same identity holds for $\vvec$. Moreover,
\[
 \begin{aligned}
 \norm{\uvec}^2
 &=\sum_{j<k<\ell}\sum_{r=1}^d
   \norm{\uvec^{(r)}_{jk\ell}}^2,
 &\norm{\vvec}^2
 &=\sum_{j<k<\ell}\sum_{r=1}^d
   \norm{\vvec^{(r)}_{jk\ell}}^2,\\
 \ip{\uvec}{\vvec}
 &=\sum_{j<k<\ell}\sum_{r=1}^d
   \ip{\uvec^{(r)}_{jk\ell}}{\vvec^{(r)}_{jk\ell}},
 &\frac{\dd\uvec\,\dd\vvec}{(\sqrt{2}\pi)^{N_2}}
 &=\prod_{j<k<\ell}\prod_{r=1}^d
   \frac{\dd\uvec^{(r)}_{jk\ell}\,\dd\vvec^{(r)}_{jk\ell}}
        {(\sqrt{2}\pi)^3}.
 \end{aligned}
\]
Thus the exponent and measure in \eqref{eq:I-def} split over the
blocks $(j,k,\ell;r)$. For a fixed block, writing
$\xvec=\uvec^{(r)}_{jk\ell}$ and
$\yvec=\vvec^{(r)}_{jk\ell}$, its exponent is
\[
 -\frac{\norm{\xvec}^2+\norm{\yvec}^2}{2}
 +\iu\ip{\xvec}{\yvec}
 +\frac{\iu\eps_d}{2}\xvec^{\mathsf T}A_{jk\ell}(\muvec)\xvec
 +\frac{\iu\eps_d}{2}\yvec^{\mathsf T}A_{jk\ell}(\muvec)\yvec.
\]
The additional quadratic terms are purely imaginary, so Fubini's
theorem applies. By \eqref{eq:determinant-formula}, each block integral
is $\phi_{\eps_d}(\mu_j,\mu_k,\mu_\ell)$, and hence
\[
 \mathcal I_{d,m}(\muvec)
 =\prod_{j<k<\ell}\prod_{r=1}^d
   \phi_{\eps_d}(\mu_j,\mu_k,\mu_\ell)
 =\Phi_{d,m}(\muvec).
\]
The factorization and \eqref{eq:contractive-factor} give
$|\mathcal I_{d,m}(\muvec)|\leq1$. Therefore
\[
 |T_{d,m,n}|
 \leq\E_{\tauvec,\sigmavec}
 \left|\mathcal I_{d,m}(\muvec)\right|
 \leq1.
\]
\end{proof}

\section{Cubic parity and quantitative estimates}\label{sec:parity}

For $\muvec\in\Z^m$, write
\begin{equation}\label{eq:e3}
 e_3(\muvec):=\sum_{1\leq j<k<\ell\leq m}\mu_j\mu_k\mu_\ell.
\end{equation}
The one-triple estimate immediately yields a uniform product estimate.

\begin{lemma}[Global determinant convergence]\label{lem:global}
If $|\mu_s|\leq L$ for every $s$, then, for $d\geq CL^3$,
\begin{equation}\label{eq:global}
 \left|\Phi_{d,m}(\muvec)-(-1)^{e_3(\muvec)}\right|
 \leq Cm^3L^4d^{-1/3}.
\end{equation}
\end{lemma}

\begin{proof}
If $|z_\alpha|,|w_\alpha|\leq1$, then
\[
 \left|\prod_\alpha z_\alpha-\prod_\alpha w_\alpha\right|
 \leq\sum_\alpha|z_\alpha-w_\alpha|.
\]
Using
$\prod_{j<k<\ell}(-1)^{\mu_j\mu_k\mu_\ell}
=(-1)^{e_3(\muvec)}$, the contractivity bound, and
\cref{cor:one-triple}, we obtain
\begin{align*}
 \left|\Phi_{d,m}(\muvec)-(-1)^{e_3(\muvec)}\right|
 &\leq\sum_{j<k<\ell}
 \left|\phi_{\eps_d}(\mu_j,\mu_k,\mu_\ell)^d
       -(-1)^{\mu_j\mu_k\mu_\ell}\right|\\
 &\leq C\binom m3L^4d^{-1/3}
 \leq Cm^3L^4d^{-1/3}.
\end{align*}
\end{proof}

For the random frequency vector in \eqref{eq:mu}, define
\begin{equation}\label{eq:E-mn}
 E_{m,n}:=\E_{\tauvec,\sigmavec}
 (-1)^{e_3(\muvec(\tauvec,\sigmavec))}.
\end{equation}
The remaining combinatorial step is to show that the cubic character
$(-1)^{e_3(\muvec)}$ agrees, with high probability, with the sign
$(-1)^n$ attached to the degree $2n+1$.

\begin{lemma}[Collision-free parity]\label{lem:collision}
For every $m\geq1$ and $n\geq0$,
\begin{equation}\label{eq:collision}
 |E_{m,n}-(-1)^n|\leq\frac{(2n+1)(2n)}m.
\end{equation}
\end{lemma}

\begin{proof}
If $n=0$, then $e_3(\muvec)=0$, so the claim is immediate. Suppose
$n\geq1$. If $m\leq2n+1$, then
\[
 |E_{m,n}-(-1)^n|\leq2
 \leq\frac{(2n+1)(2n)}m,
\]
so the claim is again immediate. We may therefore assume $m>2n+1$.
Choose $\tau_1,\ldots,\tau_{2n+1}$ independently and uniformly from
$[m]$. The exact probability of no collision is
\[
 \frac{m(m-1)\cdots(m-2n)}{m^{2n+1}}.
\]
For the proof, the simpler union bound is enough:
\[
 \Pr(\text{a collision})
 \leq\binom{2n+1}{2}\frac1m.
\]
On the collision-free event, the frequency vector has exactly $2n+1$
nonzero coordinates, all equal to $\pm1$. Modulo two, every nonzero
coordinate is $1$, and therefore
\[
 e_3(\muvec)
 \equiv\binom{2n+1}{3}\pmod2.
\]
The identity
\[
 3\binom{2n+1}{3}=n(4n^2-1)
\]
shows that $\binom{2n+1}{3}\equiv n\pmod2$. Thus
$(-1)^{e_3(\muvec)}=(-1)^n$ outside the collision event. On the
collision event the two signs differ by at most $2$, which gives
\eqref{eq:collision}.
\end{proof}

\begin{proposition}\label{prop:T-limit}
For every fixed $n\geq0$,
\begin{equation}\label{eq:T-limit}
 \lim_{m\to\infty}\lim_{d\to\infty}T_{d,m,n}=(-1)^n.
\end{equation}
\end{proposition}

\begin{proof}
Fix $m,n$. Every frequency in \eqref{eq:mu} satisfies
$|\mu_s|\leq2n+1$. By \cref{prop:block-factorization,lem:global},
\[
 \lim_{d\to\infty}T_{d,m,n}=E_{m,n}.
\]
Now let $m\to\infty$ and apply \cref{lem:collision}.
\end{proof}

The estimates above also give a useful finite-parameter statement. If
$Q\in\Z_{\geq0}$ and $L:=2Q+1$, then, for every $d\geq CL^3$ and every
$n\in\{0,\ldots,Q\}$,
\[
 |T_{d,m,n}-(-1)^n|
 \leq |T_{d,m,n}-E_{m,n}|+|E_{m,n}-(-1)^n|
 \leq Cm^3L^4d^{-1/3}+\frac{2n(2n+1)}m.
\]
Since $1=\frac2\pi\sum_{n\geq0}c_n$, splitting
\eqref{eq:coefficient-formula} at $Q$ and using
$|T_{d,m,n}-(-1)^n|\leq2$ in the tail gives
\begin{equation}\label{eq:quantitative-final}
 |\BK(f_{d,m},g_{d,m})-1|
 \leq\frac2\pi\sum_{n=0}^Q c_n
 \left(Cm^3L^4d^{-1/3}+\frac{2n(2n+1)}m\right)
 +\frac4\pi\sum_{n>Q}c_n.
\end{equation}

\section{Proof of main results}\label{sec:main-results-proof}

Now we are ready to complete the proof of \cref{thm:main}.

\begin{proof}[Proof of \cref{thm:main}]
By \cref{prop:coefficient-formula,prop:block-factorization},
$|T_{d,m,n}|\leq1$. For fixed $m$, dominated convergence in the
absolutely summable series \eqref{eq:coefficient-formula} gives
\[
 \lim_{d\to\infty}\BK(f_{d,m},g_{d,m})
 =\frac2\pi\sum_{n=0}^\infty(-1)^nc_nE_{m,n}.
\]
A second application of dominated convergence, now using
\cref{lem:collision}, yields
\[
 \lim_{m\to\infty}\lim_{d\to\infty}\BK(f_{d,m},g_{d,m})
 =\frac2\pi\sum_{n=0}^\infty c_n=1.
\]
The upper bound, the fixed-dimensional gap, and non-attainment are
proved in Appendix~\ref{app:Fourier-gap}. The same order of limits can
also be read directly from \eqref{eq:quantitative-final}: fix $Q$, first
let $d\to\infty$, then let $m\to\infty$, and finally let $Q\to\infty$.
\end{proof}

For a single diagonal choice, take $Q=Q(m)\to\infty$ sufficiently
slowly, set $L(m):=2Q(m)+1$, and require
\[
 \frac{L(m)^2}{m}\longrightarrow0,\qquad
 m^3L(m)^4d^{-1/3}\longrightarrow0.
\]
For example, take
\[
 Q(m)=\lfloor m^{1/4}\rfloor,\qquad L(m)=2Q(m)+1,\qquad d(m)=m^{13}.
\]
Then $L(m)^2/m=O(m^{-1/2})$ and
$m^3L(m)^4d(m)^{-1/3}=O(m^{-1/3})$, so both errors vanish; moreover,
$d(m)\geq CL(m)^3$ for all sufficiently large $m$.

\appendix

\section{The Fourier bound and the fixed-dimensional gap}
\label{app:Fourier-gap}

The Plancherel upper bound is standard.  BMMN's weak-compactness observation
for planar maximizers \cite[(3.1) and the paragraph following it]{BMMN}
extends verbatim to every fixed dimension, but we have not found the strict
fixed-dimensional gap stated explicitly in the literature.  Qualitatively it
follows by combining attainment with the obstruction to equality in the
Fourier bound.  Here we give a direct quantitative proof, avoiding compactness
and producing the explicit gap $\eta_N>0$ used in \cref{thm:main}.

For real $h\in L_1(\R^N)\cap L_2(\R^N)$, define the
normalized sine transform
\[
 (\mathcal Sh)(\yvec):=\frac1{(2\pi)^{N/2}}
 \int_{\R^N}h(\xvec)\sin\ip{\xvec}{\yvec}\,\dd\xvec,
\]
and let $\Podd h(\xvec)=(h(\xvec)-h(-\xvec))/2$. The sine transform
annihilates the even part and is an $L_2$ isometry on the real odd
subspace. Set
\[
 h_f(\xvec):=\pi^{-N/4}f(\xvec)e^{-\norm{\xvec}^2/2},
 \qquad
 h_g(\yvec):=\pi^{-N/4}g(\yvec)e^{-\norm{\yvec}^2/2}.
\]
Then $\norm{h_f}_2=\norm{h_g}_2=1$, and the normalization in
\eqref{eq:BK-def} gives
\begin{equation}\label{eq:Fourier-representation}
 \BK(f,g)=\ip{\mathcal S\Podd h_f}{\Podd h_g}_{L_2}
          =\ip{\mathcal S h_f}{h_g}_{L_2}.
\end{equation}
Cauchy--Schwarz and the contractivity of the odd projection give
\[
 |\BK(f,g)|
 \leq\norm{\Podd h_f}_2\norm{\Podd h_g}_2
 \leq1.
\]

If $|\BK(f,g)|=1$, then both projection norms would equal $1$. Since
$\Podd$ is an orthogonal projection, the even parts of $h_f$ and $h_g$
would vanish; hence $f$ and $g$ would be odd almost everywhere. Equality
in Cauchy--Schwarz would then give
$\mathcal Sh_f=\omega h_g$ almost everywhere for some
$\omega\in\{\pm1\}$. Since $h_f\in L_1$, its sine transform has a
continuous representative $F$, and $F(0)=0$. Taking moduli in the
almost-everywhere identity gives
\[
 |F(\yvec)|=\pi^{-N/4}e^{-\norm{\yvec}^2/2}
\]
almost everywhere and hence everywhere by continuity. This is impossible
at the origin.

The same obstruction is uniform. Since $|f|=1$ and
$|\sin\ip{\xvec}{\yvec}|\leq\norm{\xvec}\norm{\yvec}$,
\[
 |(\mathcal Sh_f)(\yvec)|
 \leq\frac{\pi^{-N/4}}{(2\pi)^{N/2}}
 \int_{\R^N}e^{-\norm{\xvec}^2/2}\norm{\xvec}\,\dd\xvec
 \norm{\yvec}
 =L_N\norm{\yvec},
\]
where
\[
 L_N:=\pi^{-N/4}\sqrt2\,
 \frac{\Gamma((N+1)/2)}{\Gamma(N/2)}.
\]
Set
\[
 r_N:=\min\left\{1,
 \frac{\pi^{-N/4}e^{-1/2}}{2L_N}\right\}>0.
\]
Since $r_N\leq1$, this choice gives
\[
 L_Nr_N\leq\tfrac12\pi^{-N/4}e^{-1/2}
 \leq\tfrac12\pi^{-N/4}e^{-r_N^2/2}.
\]
If $\norm{\yvec}\leq r_N$, then, almost everywhere,
\[
 |\mathcal Sh_f(\yvec)|
 \leq L_Nr_N
 \leq\tfrac12\pi^{-N/4}e^{-r_N^2/2}
 \leq\tfrac12\pi^{-N/4}e^{-\norm{\yvec}^2/2}
 =\tfrac12|h_g(\yvec)|.
\]
Thus the reverse triangle inequality gives, for either sign $\omega$,
\[
 |\mathcal Sh_f(\yvec)-\omega h_g(\yvec)|
 \geq\tfrac12\pi^{-N/4}e^{-\norm{\yvec}^2/2}.
\]
Set
\[
 \delta_N:=\frac14\pi^{-N/2}
 \int_{\{\norm{\yvec}\leq r_N\}}e^{-\norm{\yvec}^2}\,\dd\yvec,
 \qquad
 \eta_N:=\frac{\delta_N}{2}>0.
\]
Then $\norm{\mathcal Sh_f-\omega h_g}_2^2\geq\delta_N$, uniformly in
$f,g$. Moreover,
\[
 \norm{\mathcal Sh_f}_2=\norm{\Podd h_f}_2
 \leq\norm{h_f}_2=1.
\]
Choosing $\omega$ so that
$\omega\BK(f,g)=|\BK(f,g)|$ and using \eqref{eq:Fourier-representation},
\[
 \delta_N\leq\norm{\mathcal Sh_f-\omega h_g}_2^2
 \leq2-2|\BK(f,g)|.
\]
Hence $|\BK(f,g)|\leq1-\eta_N$.

\section{The one-dimensional whole-space theorem}
\label{app:one-dimensional-whole-space}

The one-dimensional result was announced by K\"onig~\cite{KONIG} as
joint unpublished work with Tomczak-Jaegermann. BMMN later gave a
published proof, including the uniqueness of the extremizers
\cite[Theorem~1.4 and Section~6]{BMMN}. Their argument first obtains a
maximizer on the positive half-line and then uses quantitative
oscillatory estimates to show that both maximizing sign functions are
constant there. The proof below gives a direct argument based on
positivity of an associated half-line kernel.

\begin{theorem}[Sharp one-dimensional K\"onig theorem]
\label{thm:app-one-dimensional}
For all measurable $f,g:\R\to\{\pm1\}$,
\begin{equation}\label{eq:app-one-whole-space-bilinear}
 \abs{\BK(f,g)}\leq\cK.
\end{equation}
Equality in absolute value holds if and only if there are
$\epsilon,\delta\in\{\pm1\}$ such that
\begin{equation}\label{eq:app-one-scalar-extremizers}
 f(x)=\epsilon\sgn(x),
 \qquad
 g(x)=\delta\sgn(x)
\end{equation}
almost everywhere. For such a pair,
\[
 \BK(f,g)=\epsilon\delta\,\cK.
\]
Consequently, $\BK(f,g)\leq\cK$, with equality if and only if
$f=g=\sgn$ almost everywhere or
$f=g=-\sgn$ almost everywhere.
\end{theorem}

We begin with two estimates for the half-line sine operator.

For $t>0$, put
\[
 \varphi(t):=\int_0^t e^{s^2/2}\,\dd s,
\]
and, for measurable $h:(0,\infty)\to\R$, define the unnormalized
half-line sine operator by
\[
 (S_+h)(y):=
 \int_0^\infty h(x)e^{-x^2/2}\sin(xy)\,\dd x,
 \qquad y>0.
\]
Let
\[
 \psi(y):=\int_0^\infty e^{-x^2/2}\sin(xy)\,\dd x.
\]
Differentiation under the integral and integration by parts give
$\psi'(y)+y\psi(y)=1$ and $\psi(0)=0$. Consequently,
\begin{equation}\label{eq:app-one-Gaussian-sine}
 \psi(y)=e^{-y^2/2}\varphi(y)>0.
\end{equation}
We shall use the normalization
\begin{equation}\label{eq:app-one-half-line-normalization}
 \int_0^\infty e^{-y^2/2}\psi(y)\,\dd y
 =
 \int_0^\infty e^{-y^2}\varphi(y)\,\dd y
 =\frac{\pi}{2\sqrt2}\cK.
\end{equation}
Its verification is given in \cref{lem:app-one-constant} below.

\begin{lemma}[Strict Dirichlet positivity]
\label{lem:app-one-Dirichlet}
Suppose that $\chi:(0,\infty)\to(0,\infty)$ is strictly decreasing and
integrable. Then, for every $u>0$,
\[
 \int_0^\infty\chi(t)\sin(ut)\,\dd t>0.
\]
\end{lemma}

\begin{proof}
After the change of variables $s=ut$, pair every positive half-wave of
the sine with the following negative half-wave:
\begin{align*}
 \int_0^\infty\chi(t)\sin(ut)\,\dd t
 &=
 \frac1u\sum_{k=0}^\infty\int_0^\pi
 \left[
 \chi\left(\frac{2k\pi+s}{u}\right)
 -
 \chi\left(\frac{(2k+1)\pi+s}{u}\right)
 \right]\sin s\,\dd s>0.
\end{align*}
The series is absolutely convergent because $\chi$ is integrable.
\end{proof}

\begin{lemma}[Sharp scalar half-line sine estimate]
\label{lem:app-one-half-line}
Let $h:(0,\infty)\to[-1,1]$ be measurable. Then
\begin{equation}\label{eq:app-one-quadratic-bound}
 \int_0^\infty
 \frac{(S_+h(y))^2}{\varphi(y)}\,\dd y
 \leq\frac{\pi}{2\sqrt2}\cK
\end{equation}
and
\begin{equation}\label{eq:app-one-half-line-bound}
 \int_0^\infty
 e^{-y^2/2}\abs{S_+h(y)}\,\dd y
 \leq\frac{\pi}{2\sqrt2}\cK.
\end{equation}
Equality holds in either inequality if and only if
$h=1$ almost everywhere or $h=-1$ almost everywhere.
\end{lemma}

\begin{proof}
For $x,z>0$, let
\[
 K_+(x,z):=
 \int_0^\infty
 \frac{\sin(xy)\sin(zy)}{\varphi(y)}\,\dd y.
\]
We first prove that $K_+(x,z)>0$. Put
\[
 \alpha:=\abs{x-z},\qquad \beta:=x+z.
\]
Since $\varphi'(y)=e^{y^2/2}$, the derivative of
$\chi(y):=y/\varphi(y)$ is
\[
 \chi'(y)
 =\frac{\varphi(y)-ye^{y^2/2}}{\varphi(y)^2}<0.
\]
Indeed,
\[
 ye^{y^2/2}-\varphi(y)
 =
 \int_0^y\left(e^{y^2/2}-e^{t^2/2}\right)\,\dd t>0.
\]
Continuity of the integrand defining $\varphi$ gives
$\varphi(y)/y\to1$ as $y\downarrow0$. At infinity, l'H\^opital's rule
gives
\[
 \lim_{y\to\infty}
 \frac{\varphi(y)}{e^{y^2/2}/y}
 =\lim_{y\to\infty}
 \frac{e^{y^2/2}}{e^{y^2/2}(1-y^{-2})}=1.
\]
Hence $\chi(y)\to1$ as $y\downarrow0$ and
$\chi(y)\sim y^2e^{-y^2/2}$ as $y\to\infty$, proving that
$\chi\in L_1(0,\infty)$. The product-to-sum identity and
\[
 \cos(\alpha y)-\cos(\beta y)
 =y\int_\alpha^\beta\sin(ty)\,\dd t
\]
therefore give
\begin{align*}
 2K_+(x,z)
 &=\int_0^\infty
 \frac{\cos((x-z)y)-\cos((x+z)y)}{\varphi(y)}\,\dd y\\
 &=\int_\alpha^\beta\left(
   \int_0^\infty\chi(y)\sin(ty)\,\dd y
 \right)\dd t.
\end{align*}
Here Fubini's theorem applies because
\[
 \int_\alpha^\beta\int_0^\infty
 \chi(y)\abs{\sin(ty)}\,\dd y\,\dd t
 \leq(\beta-\alpha)\int_0^\infty\chi(y)\,\dd y<\infty.
\]
By \cref{lem:app-one-Dirichlet}, the inner integral is strictly positive
for every $t>0$. Since $\beta>\alpha\geq0$, we obtain
\begin{equation}\label{eq:app-one-kernel-positive}
 K_+(x,z)>0,\qquad x,z>0.
\end{equation}

The quadratic expansion below is justified by
\[
 \int_0^\infty\frac1{\varphi(y)}
 \left(
   \int_0^\infty e^{-x^2/2}\abs{\sin(xy)}\,\dd x
 \right)^2\dd y<\infty.
\]
To verify this, dominated convergence, using
$\abs{\sin(xy)}/y\leq x$, gives
\[
 \lim_{y\downarrow0}\frac1y
 \int_0^\infty e^{-x^2/2}\abs{\sin(xy)}\,\dd x
 =\int_0^\infty xe^{-x^2/2}\,\dd x=1.
\]
Together with $\varphi(y)\sim y$, this shows that the integrand in the
preceding display is asymptotic to $y$ at the origin. On the other hand,
\[
 \int_0^\infty e^{-x^2/2}\abs{\sin(xy)}\,\dd x
 \leq\int_0^\infty e^{-x^2/2}\,\dd x=\sqrt{\frac\pi2},
\]
and $1/\varphi(y)\sim ye^{-y^2/2}$ at infinity. Thus the same integrand
is $O(ye^{-y^2/2})$ there, proving the claimed integrability. Fubini's
theorem gives
\begin{align*}
 \int_0^\infty\frac{(S_+h(y))^2}{\varphi(y)}\,\dd y
 &=
 \int_0^\infty\int_0^\infty
 e^{-(x^2+z^2)/2}K_+(x,z)
 h(x)h(z)\,\dd x\,\dd z\\
 &\leq
 \int_0^\infty\int_0^\infty
 e^{-(x^2+z^2)/2}K_+(x,z)\,\dd x\,\dd z\\
 &=
 \int_0^\infty
 \frac{\psi(y)^2}{\varphi(y)}\,\dd y
 =\frac{\pi}{2\sqrt2}\cK.
\end{align*}
This proves \eqref{eq:app-one-quadratic-bound}.

If equality holds, strict positivity in
\eqref{eq:app-one-kernel-positive} forces
\[
 h(x)h(z)=1
\]
for almost every $(x,z)\in(0,\infty)^2$. Hence
$h=1$ almost everywhere or $h=-1$ almost everywhere.

For \eqref{eq:app-one-half-line-bound}, weighted Cauchy--Schwarz gives
\begin{align*}
 \int_0^\infty e^{-y^2/2}\abs{S_+h(y)}\,\dd y
 &\leq
 \left(
   \int_0^\infty e^{-y^2/2}\psi(y)\,\dd y
 \right)^{1/2}
 \left(
   \int_0^\infty
   \frac{e^{-y^2/2}(S_+h(y))^2}{\psi(y)}
   \,\dd y
 \right)^{1/2}\\
 &=\left(\frac{\pi}{2\sqrt2}\cK\right)^{1/2}
 \left(
   \int_0^\infty
   \frac{(S_+h(y))^2}{\varphi(y)}\,\dd y
 \right)^{1/2}
 \leq\frac{\pi}{2\sqrt2}\cK.
\end{align*}
Equality forces equality in \eqref{eq:app-one-quadratic-bound}.
Conversely, the definition of $\psi$ shows that both constant
inputs $h=1$ and $h=-1$ attain equality.
\end{proof}

\begin{lemma}[Evaluation of the half-line integral]
\label{lem:app-one-constant}
The normalization identity
\eqref{eq:app-one-half-line-normalization} holds.
\end{lemma}

\begin{proof}
Using polar coordinates
$x=r\cos\theta$, $y=r\sin\theta$, followed by $s=r^2/2$, gives
\begin{align*}
 \int_0^\infty e^{-y^2/2}\psi(y)\,\dd y
 &=
 \int_0^\infty\int_0^\infty
 e^{-(x^2+y^2)/2}\sin(xy)\,\dd x\,\dd y\\
 &=
 \int_0^{\pi/2}\int_0^\infty
 e^{-s}\sin\bigl(s\sin(2\theta)\bigr)\,\dd s\,\dd\theta\\
 &=
 \int_0^{\pi/2}
 \frac{\sin(2\theta)}{1+\sin^2(2\theta)}\,\dd\theta
 =
 \frac12\int_{-1}^1\frac{\dd u}{2-u^2}\\
 &=
 \frac1{\sqrt2}\log(1+\sqrt2)
 =\frac{\pi}{2\sqrt2}\cK.
\end{align*}
The second integral in \eqref{eq:app-one-half-line-normalization}
equals the first by \eqref{eq:app-one-Gaussian-sine}.
\end{proof}

\begin{proof}[Proof of \cref{thm:app-one-dimensional}]
Since the sine kernel is odd in each variable, folding both integrals
onto $(0,\infty)$ gives
\[
 \BK(f,g)
 =
 \frac4{\sqrt{2}\pi}
 \int_0^\infty e^{-y^2/2}
 \frac{g(y)-g(-y)}2
 S_+\!\left(\frac{f(\,\cdot\,)-f(-\,\cdot\,)}2\right)(y)\,\dd y.
\]
Both difference quotients take values in $[-1,1]$. Hence
\cref{lem:app-one-half-line} gives
\begin{align*}
 \abs{\BK(f,g)}
 &\leq
 \frac4{\sqrt{2}\pi}
 \int_0^\infty e^{-y^2/2}
 \abs{S_+\!\left(\frac{f(\,\cdot\,)-f(-\,\cdot\,)}2\right)(y)}
 \,\dd y\\
 &\leq
 \frac4{\sqrt{2}\pi}\frac{\pi}{2\sqrt2}\cK
 =
 \cK.
\end{align*}

Suppose equality holds in absolute value. Equality in
\eqref{eq:app-one-half-line-bound} implies that
$\bigl(f(x)-f(-x)\bigr)/2=\epsilon$ for almost every $x>0$ and some
$\epsilon\in\{\pm1\}$. The sine transform in the display above is then
$\epsilon\psi$. Since $\psi>0$, equality in the first inequality forces
$\bigl(g(y)-g(-y)\bigr)/2=\delta$ for almost every $y>0$, for some fixed
$\delta\in\{\pm1\}$.

Because $f$ and $g$ are Boolean, the identities
\[
 \frac{f(x)-f(-x)}2=\epsilon,
 \qquad
 \frac{g(x)-g(-x)}2=\delta
\]
force
$f(x)=\epsilon\sgn(x)$ and
$g(x)=\delta\sgn(x)$ almost everywhere. Conversely, every such pair
attains equality, with sign $\epsilon\delta$.
\end{proof}

\section{The two-dimensional quadrant theorem}
\label{app:two-dimensional-quadrant}

The following result is the two-dimensional counterpart of the
half-line estimate in \cref{app:one-dimensional-whole-space}. Its proof
combines the sharp half-line sine bound with an elementary cosine
estimate and the tensor decomposition of the quadrant sine operator.

Put
\[
 \mathsf Q_2:=[0,\infty)^2.
\]
For a measurable function $h:\mathsf Q_2\to[-1,1]$, define the
unnormalized quadrant sine operator
\[
 (S_{\mathsf Q_2}h)(\yvec)
 :=
 \int_{\mathsf Q_2}
 h(\xvec)e^{-\norm{\xvec}^2/2}\sin\ip{\xvec}{\yvec}\,\dd\xvec,
 \qquad \yvec\in\mathsf Q_2.
\]

\begin{theorem}[Sharp scalar quadrant bound]
\label{thm:app-quadrant}
For every measurable $h:\mathsf Q_2\to[-1,1]$,
\begin{equation}\label{eq:app-quadrant-main-inequality}
 \int_{\mathsf Q_2}e^{-\norm{\yvec}^2/2}\abs{(S_{\mathsf Q_2}h)(\yvec)}\,\dd\yvec\leq\frac{\pi^2}{4}\cK.
\end{equation}
Equality is attained by the constant inputs $h=1$ and $h=-1$.
\end{theorem}

We begin with the auxiliary half-line cosine estimate.

In addition to $S_+$ from
\cref{app:one-dimensional-whole-space}, define the unnormalized
half-line cosine operator
\[
 (C_+u)(y):=
 \int_0^\infty u(x)e^{-x^2/2}\cos(xy)\,\dd x.
\]
Then
\begin{equation}\label{eq:app-quadrant-basic-transforms}
 \psi(y)=e^{-y^2/2}\varphi(y),
 \qquad
 (C_+\mathbf 1)(y)=\sqrt{\frac\pi2}e^{-y^2/2}.
\end{equation}
The full-line complex combination $C+\iu S$ is, after the usual
Gaussian rescaling and weight, the complex-time Hermite semigroup at
$z=\iu$; its even and odd restrictions give $C$ and $S$, respectively.
We do not use this interpretation below.
\begin{lemma}[Sharp scalar half-line cosine estimate]
\label{lem:app-quadrant-cosine}
For every bounded measurable $u:(0,\infty)\to\R$,
\begin{equation}\label{eq:app-quadrant-cosine-L2}
 \int_0^\infty\abs{C_+u(y)}^2\,\dd y
 =
 \frac\pi2\int_0^\infty e^{-x^2}\abs{u(x)}^2\,\dd x.
\end{equation}
If, in addition, $\abs{u}\leq1$, then
\begin{equation}\label{eq:app-quadrant-cosine-L1}
 \int_0^\infty e^{-y^2/2}\abs{C_+u(y)}\,\dd y
 \leq\frac{\pi}{2\sqrt2}.
\end{equation}
Equality is attained by $u=1$ and $u=-1$.
\end{lemma}

\begin{proof}
Let $u_e(t):=u(\abs t)$. With the Fourier convention
$\widehat h(y)=\int_\R h(x)e^{\iu xy}\,\dd x$,
the Fourier transform below is even and equals $2C_+u$ on
$(0,\infty)$. Hence Plancherel gives
\[
 8\int_0^\infty\abs{C_+u(y)}^2\,\dd y
 =
 \int_\R\abs{\widehat{u_e(\,\cdot\,)e^{-(\,\cdot\,)^2/2}}(y)}^2\,\dd y
 =
 4\pi\int_0^\infty e^{-x^2}\abs{u(x)}^2\,\dd x,
\]
which is \eqref{eq:app-quadrant-cosine-L2}.
When $\abs{u}\leq1$, Cauchy--Schwarz yields
\begin{align*}
 \int_0^\infty e^{-y^2/2}\abs{C_+u(y)}\,\dd y
 &\leq
 \left(\int_0^\infty e^{-y^2}\,\dd y\right)^{1/2}
 \left(\int_0^\infty\abs{C_+u(y)}^2\,\dd y\right)^{1/2}\\
 &\leq
 \left(\frac{\sqrt\pi}{2}\right)^{1/2}
 \left(\frac\pi2\frac{\sqrt\pi}{2}\right)^{1/2}
 =\frac{\pi}{2\sqrt2}.
\end{align*}
For $u=\pm1$,
$C_+u(y)=\pm\sqrt{\pi/2}\,e^{-y^2/2}$, so equality holds.
\end{proof}

\begin{remark}[Hilbert-valued extension]
\label{rem:app-quadrant-Hilbert-extension}
The two half-line sine estimates in
\eqref{eq:app-one-quadratic-bound} and
\eqref{eq:app-one-half-line-bound} extend to functions with values in
a real Hilbert space and pointwise norm at most $1$. Indeed, the
quadratic estimate uses only
\[
 \ip{h(x)}{h(z)}\leq1.
\]
The $L_1$ estimate then follows by the same Cauchy--Schwarz argument.
The proof of \cref{lem:app-quadrant-cosine} likewise applies to
Hilbert-valued functions, using Hilbert-valued Plancherel and
Cauchy--Schwarz. Consequently, the proof below also yields the quadrant
bound for Hilbert-valued inputs, with the same constant and with
equality for a constant function in a fixed unit direction.
\end{remark}

\begin{proof}[Proof of \cref{thm:app-quadrant}]
The identity
\[
 \sin(x_1y_1+x_2y_2)
 =
 \sin(x_1y_1)\cos(x_2y_2)
 +\cos(x_1y_1)\sin(x_2y_2)
\]
and absolute convergence of the Gaussian integrals give
\begin{equation}\label{eq:app-quadrant-decomposition}
 S_{\mathsf Q_2}=S_+\otimes C_++C_+\otimes S_+.
\end{equation}
For $h=1$,
\[
 S_{\mathsf Q_2}h(\yvec)
 =
 \psi(y_1)(C_+\mathbf 1)(y_2)
 +(C_+\mathbf 1)(y_1)\psi(y_2).
\]
Both summands are positive on the interior of $\mathsf Q_2$. Hence
\begin{align*}
 \int_{\mathsf Q_2}
 e^{-\norm{\yvec}^2/2}\abs{S_{\mathsf Q_2}h(\yvec)}\,\dd\yvec
 &=
 2\left(
   \int_0^\infty e^{-y^2/2}\psi(y)\,\dd y
 \right)
 \left(
   \int_0^\infty e^{-y^2/2}(C_+\mathbf 1)(y)\,\dd y
 \right)\\
 &=\frac{\pi^2}{4}\cK.
\end{align*}
The input $h=-1$ gives the same value by linearity and the outer
absolute value.

For the upper bound, we use the same tensor decomposition. Let
$h,k:\mathsf Q_2\to[-1,1]$ be measurable, and denote by
$\mathcal B_{SC}$ the bilinear form induced by the first summand:
\[
 \mathcal B_{SC}(h,k)
 :=
 \int_{\mathsf Q_2}e^{-\norm{\yvec}^2/2}
 ((S_+\otimes C_+)h)(\yvec)k(\yvec)\,\dd\yvec.
\]
Define $\mathcal B_{CS}$ similarly with $C_+\otimes S_+$ in place of
$S_+\otimes C_+$.
Let $\mathcal H:=L_2((0,\infty),\dd t)$. For $x_2,y_2>0$, define
vectors $\mathsf U(x_2),\mathsf V(y_2)\in\mathcal H$ by
\begin{align*}
 (\mathsf U(x_2))(t)
 &:=
 \frac{S_+(h(\cdot,x_2))(t)}
      {\sqrt{\varphi(t)}},\\
 (\mathsf V(y_2))(t)
 &:=
 e^{-t^2/2}\sqrt{\varphi(t)}\,k(t,y_2).
\end{align*}
By \eqref{eq:app-one-quadratic-bound} and
\eqref{eq:app-one-half-line-normalization}, respectively,
\[
 \norm{\mathsf U(x_2)}_{\mathcal H}^2
 \leq\frac{\pi}{2\sqrt2}\cK,
 \qquad
 \norm{\mathsf V(y_2)}_{\mathcal H}^2
 \leq\frac{\pi}{2\sqrt2}\cK.
\]
Thus the first-coordinate sine form has the fixed factorization
\[
 \mathcal B_{SC}(h,k)
 =
 \int_0^\infty\int_0^\infty
 e^{-(x_2^2+y_2^2)/2}\cos(x_2y_2)
 \ip{\mathsf U(x_2)}{\mathsf V(y_2)}_{\mathcal H}
 \,\dd x_2\,\dd y_2.
\]
The two vectors in the inner product depend separately on $h$ and
$k$. This is the tensorization point: the remaining coordinate is a
cosine form with inputs in the fixed Hilbert space $\mathcal H$.

Applying the Hilbert-valued form of
\cref{lem:app-quadrant-cosine} from
\cref{rem:app-quadrant-Hilbert-extension} and using homogeneity, we obtain
the following bound for the mixed tensor term by the product of the
one-dimensional sine and cosine constants:
\[
 \abs{\mathcal B_{SC}(h,k)}
 \leq
 \frac{\pi}{2\sqrt2}
 \operatorname*{ess\,sup}_{x_2>0}\norm{\mathsf U(x_2)}_{\mathcal H}
 \operatorname*{ess\,sup}_{y_2>0}\norm{\mathsf V(y_2)}_{\mathcal H}
 \leq
 \left(\frac{\pi}{2\sqrt2}\cK\right)
 \left(\frac{\pi}{2\sqrt2}\right)
 =\frac{\pi^2}{8}\cK.
\]
The pointwise Hilbert-space bounds above also justify the factorization
and Fubini's theorem by Cauchy--Schwarz and Gaussian integrability.

Interchanging the two coordinates gives the same bound for the
bilinear form $\mathcal B_{CS}$, with the two one-dimensional factors
in the opposite order. Therefore the triangle inequality gives
\[
 \abs{\mathcal B_{SC}(h,k)+\mathcal B_{CS}(h,k)}
 \leq
 \left(\frac{\pi}{2\sqrt2}\cK\right)
 \left(\frac{\pi}{2\sqrt2}\right)
 +
 \left(\frac{\pi}{2\sqrt2}\right)
 \left(\frac{\pi}{2\sqrt2}\cK\right)
 =\frac{\pi^2}{4}\cK.
\]
By \eqref{eq:app-quadrant-decomposition}, the expression on the left
is the bilinear pairing of $S_{\mathsf Q_2}h$ with $k$. Taking
$k=\sgn(S_{\mathsf Q_2}h)$ proves the theorem.
\end{proof}

\begin{remark}
The same iterated tensor argument gives bounds in higher-dimensional
quadrants; by contrast, on the full space $\R^2$ one has
$\cK<\KoDim{2}<1$ by BMMN~\cite{BMMN} and \cref{thm:main}.
\end{remark}

\newcommand{\etalchar}[1]{$^{#1}$}


\begin{thebibliography}{BMMN13}

\bibitem[AN06]{ALONNAOR}
Noga Alon and Assaf Naor.
\newblock Approximating the cut-norm via {Grothendieck}'s inequality.
\newblock {\em SIAM Journal on Computing}, 35(4):787--803, 2006.

\bibitem[BMMN13]{BMMN}
Mark Braverman, Konstantin Makarychev, Yury Makarychev, and Assaf Naor.
\newblock The {Grothendieck} constant is strictly smaller than {Krivine}'s
  bound.
\newblock {\em Forum of Mathematics, Pi}, 1:e4, 2013.

\bibitem[Dav85]{DAVIE}
A.~M. Davie.
\newblock Matrix norms related to {Grothendieck}'s inequality.
\newblock In {\em Banach Spaces (Columbia, Mo., 1984)}, volume~1166 of
  {\em Lecture Notes in Mathematics}, pages 22--26. Springer, Berlin, 1985.

\bibitem[Fol89]{FOLLAND}
Gerald~B. Folland.
\newblock {\em Harmonic Analysis in Phase Space}, volume 122 of {\em Annals of
  Mathematics Studies}.
\newblock Princeton University Press, Princeton, NJ, 1989.

\bibitem[Gro53]{GROTH}
Alexander Grothendieck.
\newblock R\'esum\'e de la th\'eorie m\'etrique des produits tensoriels
  topologiques.
\newblock {\em Boletim da Sociedade de Matem\'atica de S\~ao Paulo}, 8:1--79,
  1953.

\bibitem[Haa87]{HAAGERUP}
Uffe Haagerup.
\newblock A new upper bound for the complex {Grothendieck} constant.
\newblock {\em Israel Journal of Mathematics}, 60(2):199--224, 1987.

\bibitem[Hei26a]{HEILMANLOWER}
Steven Heilman.
\newblock A lower bound for {Grothendieck}'s constant, 2026.
\newblock Preprint, arXiv:2603.22616.

\bibitem[Hei26b]{HEILMAN}
Steven Heilman.
\newblock An upper bound on {Grothendieck}'s constant, 2026.
\newblock Preprint, arXiv:2606.00247.

\bibitem[JM26]{JM}
Chris Jones and Giulio Malavolta.
\newblock The {Grothendieck} constant is strictly larger than {Davie--Reeds}'
  bound, 2026.
\newblock Preprint, arXiv:2603.30039.

\bibitem[K{\"o}n01]{KONIG}
Hermann K{\"o}nig.
\newblock On an extremal problem originating in questions of unconditional
  convergence.
\newblock In {\em Recent Progress in Multivariate Approximation}, volume 137 of
  {\em International Series of Numerical Mathematics}, pages 185--192.
  Birkh{\"a}user, Basel, 2001.

\bibitem[Kri77]{KRIVINE}
Jean-Louis Krivine.
\newblock Sur la constante de {Grothendieck}.
\newblock {\em Comptes Rendus de l'Acad\'emie des Sciences de Paris, S\'erie
  A--B}, 284(8):A445--A446, 1977.

\bibitem[Kri79]{KRIVINE79}
Jean-Louis Krivine.
\newblock Constantes de {Grothendieck} et fonctions de type positif sur les
  sph\`eres.
\newblock {\em Advances in Mathematics}, 31(1):16--30, 1979.

\bibitem[Kri23]{KRIVINENOTE}
Jean-Louis Krivine.
\newblock A note about {Grothendieck}'s constant, 2023.
\newblock Preprint, arXiv:2306.00995.

\bibitem[KN12]{KHOTNAOR}
Subhash Khot and Assaf Naor.
\newblock {Grothendieck}-type inequalities in combinatorial optimization.
\newblock {\em Communications on Pure and Applied Mathematics},
  65(7):992--1035, 2012.

\bibitem[LSX{\etalchar{+}}26]{LISK}
Alan Li, Rahul Saha, Anton Xue, Swarat Chaudhuri, Adam Klivans, Pravesh~K.
  Kothari, and Raghu Meka.
\newblock The {Grothendieck} constant is less than
  {$\pi/(2\log(1+\sqrt{2}))-10^{-5}$}, 2026.
\newblock Preprint, arXiv:2606.03991v2.

\bibitem[NR14]{NR}
Assaf Naor and Oded Regev.
\newblock {Krivine} schemes are optimal.
\newblock {\em Proceedings of the American Mathematical Society},
  142:4315--4320, 2014.

\bibitem[NRV14]{NRV}
Assaf Naor, Oded Regev, and Thomas Vidick.
\newblock Efficient rounding for the noncommutative {Grothendieck}
  inequality.
\newblock {\em Theory of Computing}, 10(11):257--295, 2014.

\bibitem[Pis86]{PISIER}
Gilles Pisier.
\newblock {\em Factorization of Linear Operators and Geometry of Banach
  Spaces}.
\newblock Volume~60 of {\em CBMS Regional Conference Series in Mathematics}.
\newblock American Mathematical Society, Providence, RI, 1986.

\bibitem[Ree91]{REEDS}
J.~A. Reeds.
\newblock A new lower bound on the real {Grothendieck} constant.
\newblock Unpublished manuscript, 1991.

\bibitem[SLX{\etalchar{+}}26]{SLXCKKM}
Rahul Saha, Alan Li, Anton Xue, Swarat Chaudhuri, Adam Klivans, Pravesh~K.
  Kothari, and Raghu Meka.
\newblock New lower and upper bounds for the {Grothendieck} constant, 2026.
\newblock Preprint, arXiv:2608.11158v2.

\bibitem[Ver26]{VERSHYNIN}
Roman Vershynin.
\newblock {\em High-Dimensional Probability: An Introduction with
  Applications in Data Science}.
\newblock Second edition, Cambridge University Press, Cambridge, 2026.

\end{thebibliography}
\end{document}